\documentclass[11pt]{article} 

\usepackage[parfill]{parskip} 
\usepackage{graphicx}
\usepackage{algorithm,algorithmic}
\usepackage{amsmath, amssymb, amsthm, bm,bbm}
\usepackage{url}
\usepackage{fullpage, prettyref}
\usepackage{boxedminipage}
\usepackage{wrapfig}
\usepackage{ifthen}
\usepackage{xcolor}
\usepackage{framed}
\usepackage{algorithmic,algorithm}
\usepackage{enumitem}
\newtheorem{theorem}{Theorem}

\newtheorem{lemma}{Lemma}
\newtheorem{claim}{Claim}

\newtheorem{definition}{Definition}

\newtheorem{assumption}{Assumption}

        {\hspace*{\fill}$\Box$\par}

\newcommand{\ignore}[1]{}
\renewcommand{\cal}[1]{\mathcal #1}

\newcommand{\R}{\mathbb{R}}

\newcommand{\poly}{\mathrm{poly}}

\newcommand{\wh}{\widehat}

\newcommand{\cP}{{\cal P}}

\newcommand{\cT}{{\cal T}}

\newcommand{\set}[1]{\left\{ #1 \right\}}
\newcommand{\ind}[1]{\mathbbm{1}_{\{#1\}}}
\newcommand{\ones}{\mathbbm{1}}
\newcommand{\abs}[1]{\left\lvert #1 \right\rvert}
\newcommand{\norm}[1]{\left\lVert #1 \right\rVert}
\newcommand{\infnorm}[1]{\left\lVert #1 \right\rVert_{\infty}}
\newcommand{\twonorm}[1]{\left\lVert #1 \right\rVert_{2}}
\newcommand{\frob}[1]{\left\lVert #1 \right\rVert_{F}}
\newcommand{\ceil}[1]{\lceiling #1 \rceiling}

\newcommand{\ip}[2]{{\left\langle #1, #2 \right\rangle}}
\newcommand{\defas}{:=}
\newcommand{\balpha}{\pmb{\alpha}}

\renewcommand{\Pr}{\mathbf{Pr}}
\newcommand{\prob}[1]{\Pr\left[#1\right]}
\newcommand{\expec}[1]{\mathbb{E}\left[#1\right]}

\newcommand{\Sec}[1]{\hyperref[sec:#1]{\S\ref*{sec:#1}}} 
\newcommand{\Eqn}[1]{\hyperref[eq:#1]{(\ref*{eq:#1})}} 
\newcommand{\Fig}[1]{\hyperref[fig:#1]{Fig.\,\ref*{fig:#1}}} 
\newcommand{\Tab}[1]{\hyperref[tab:#1]{Tab.\,\ref*{tab:#1}}} 
\newcommand{\Thm}[1]{\hyperref[thm:#1]{Theorem\,\ref*{thm:#1}}} 
\newcommand{\Lem}[1]{\hyperref[lem:#1]{Lemma\,\ref*{lem:#1}}} 
\newcommand{\Prop}[1]{\hyperref[prop:#1]{Prop.~\ref*{prop:#1}}} 
\newcommand{\Cor}[1]{\hyperref[cor:#1]{Corollary~\ref*{cor:#1}}} 
\newcommand{\Def}[1]{\hyperref[def:#1]{Definition~\ref*{def:#1}}} 
\newcommand{\Alg}[1]{\hyperref[alg:#1]{Alg.~\ref*{alg:#1}}} 
\newcommand{\Ex}[1]{\hyperref[ex:#1]{Ex.~\ref*{ex:#1}}} 
\newcommand{\Clm}[1]{\hyperref[clm:#1]{Claim~\ref*{clm:#1}}} 
\colorlet{shadecolor}{blue!20}

\newcommand{\wH}{\widehat{\Hmat}}
\newcommand{\wX}{\widehat{X}}

\title{Fast Exact Matrix Completion with Finite Samples}

\author{ Prateek Jain\thanks{Microsoft Research, India. Email: prajain@microsoft.com} \and 
Praneeth Netrapalli\thanks{Microsoft Research, Cambridge MA. Email: praneeth@microsoft.com (Part of the work done while a student at UT Austin and interning at MSR India.)}
}

\renewcommand{\S}{S}
\newcommand{\E}{\mathbb{E}}
\newcommand{\Em}{\mathbf{E}}
\newcommand{\Fm}{\mathbf{F}}

\newcommand{\X}{\bm{X}}
\newcommand{\Xplus}{\bm{X_+}}

\newcommand{\Pk}[1][k]{P_{#1}}
\newcommand{\Lo}{\bm{\Lambda}}
\newcommand{\So}{\mathbf{\Sigma}}
\newcommand{\Sh}{\mathbf{\widehat{\Sigma}}}

\newcommand{\Us}{\U^*}
\newcommand{\Vs}{\V^*}
\newcommand{\Ustr}{\trans{\Us}}
\newcommand{\Vstr}{\trans{\Vs}}
\newcommand{\Usperp}{\U^*_{\perp}}

\newcommand{\Usperptr}{\trans{\Usperp}}

\newcommand{\U}{\bm{U}}
\newcommand{\V}{\bm{V}}
\newcommand{\Utr}{\trans{\U}}
\newcommand{\Vtr}{\trans{\V}}
\newcommand{\Util}{\widetilde{\U}}
\newcommand{\Utiltr}{\trans{\Util}}
\newcommand{\Uperp}{{\U}_{\perp}}
\newcommand{\Uperptr}{\trans{\Uperp}}
\newcommand{\Lhat}{\bm{\widehat{\Lambda}}}
\newcommand{\lami}[1][i]{\lambda_{#1}}
\newcommand{\lhati}[1][i]{\widehat{\lambda}_{#1}}
\newcommand{\ltili}[1][i]{\widetilde{\lambda}_{#1}}
\newcommand{\M}{\mathbf{M}}

\newcommand{\whM}{\mathbf{\widehat{M}}}
\newcommand{\Sob}{\underline{\mathbf{\Sigma}}}

\newcommand{\Mb}{\underline{\M}}
\newcommand{\PkM}{\Pk\left(\M\right)}
\newcommand{\cPT}[1]{\cP_{\cT}\left({#1}\right)}
\newcommand{\cPTp}[1]{\cP_{\cT^{\perp}}\left({#1}\right)}

\newcommand{\Nmat}{\bm{N}}
\newcommand{\A}{\bm{A}}
\newcommand{\ei}{\bm{e}_i}
\newcommand{\ej}{\bm{e}_j}
\newcommand{\er}{\bm{e}_r}
\newcommand{\W}{\bm{W}}

\newcommand{\mB}{\bm{B}}
\newcommand{\C}{\bm{C}}
\newcommand{\G}{\bm{G}}
\newcommand{\mE}{\bm{E}}
\newcommand{\Hmat}{\mathbf{H}}
\newcommand{\eye}{\mathbf{I}}
\newcommand{\linf}{\ell_{\infty}}
\newcommand{\ltwo}{\ell_{2}}

\newcommand{\Mto}[1][t]{\X_{#1}}

\newcommand{\Mt}[1][t]{\X_{k,#1}}

\newcommand{\Xkpt}[1][t]{\X_{k+1,#1}}

\newcommand{\sk}[1][k]{\sigma_{#1}}
\newcommand{\e}{\bm{e}}
\newcommand{\PoA}{P_{\Omega}\left(\A\right)}
\newcommand{\Uo}{\mathbf{U}^*}
\newcommand{\Uotr}{\left(\mathbf{U}^*\right)^{\top}}
\newcommand{\so}{\sigma}
\newcommand{\uo}{\bm{u}^*}
\newcommand{\uoi}[1][i]{\uo_{#1}}
\newcommand{\uoitr}[1][i]{\trans{\uo_{#1}}}
\newcommand{\ui}[1][i]{\u_{#1}}
\newcommand{\uitr}[1][i]{\trans{\u_{#1}}}

\newcommand{\trans}[1]{{#1}^{\top}}
\newcommand{\inv}[1]{{#1}^{-1}}

\renewcommand{\ceil}[1]{\lceil{#1}\rceil}
\newcommand{\gkt}[1][t]{\gamma_{k,#1}}
\newcommand{\order}[1]{O\left({#1}\right)}

\newcommand{\muzero}{\mu}

\renewcommand{\u}{\bm{u}}
\newcommand{\x}{\bm{x}}
\newcommand{\Om}[1]{P_{\Omega}\left(#1\right)}
\renewcommand{\poly}[1]{\textrm{poly}\left(#1\right)}
\newcommand{\svp}{SVP}
\newcommand{\stsvp}{St-SVP}
\newcommand{\pgd}{PGD}
\newcommand{\gd}{GD}
\newcommand{\hT}{\widehat{T}}

\newcommand{\red}[1]{{\color{black} {#1}}}
\date{}

\begin{document}

\maketitle

\begin{abstract}
Matrix completion is the problem of recovering a low rank matrix by observing a small fraction of its entries. A series of recent works \cite{Keshavan2012,JainNS2013,Hardt2013} have proposed fast non-convex optimization based iterative algorithms to solve this problem. However, the sample complexity in all these results is sub-optimal in its dependence on the rank, condition number and the desired accuracy.

In this paper, we present a fast iterative algorithm that solves the matrix completion problem by observing $\order{nr^5 \log^3 n}$ entries, which is independent of the condition number and the desired accuracy. The run time of our algorithm is $\order{nr^7\log^3 n\log 1/\epsilon}$ which is near linear in the dimension of the matrix. To the best of our knowledge, this is the first near linear time algorithm for exact matrix completion with finite sample complexity (i.e. independent of $\epsilon$). 
Our algorithm is based on a well known projected gradient descent method, where the projection is onto the (non-convex) set of low rank matrices. There are two key ideas in our result: 1) our argument is based on a $\ell_\infty$ norm potential function (as opposed to the spectral norm) and provides a novel way to obtain perturbation bounds for it. 2) we prove and use a natural extension of the Davis-Kahan theorem to obtain perturbation bounds on the best low rank approximation of matrices with good eigen gap. Both of these ideas may be of independent interest.
\end{abstract}
\newpage
\section{Introduction}
In this paper, we study the problem of low-rank matrix completion (LRMC) where the goal is to recover a low-rank matrix
by observing a tiny fraction of its entries. That is, given ${\cal M}=\{M_{ij}, (i,j)\in \Omega\}$, where $\M\in \R^{n_1\times n_2}$ is an unknown rank-$r$ matrix and $\Omega \subseteq [n_1] \times [n_2]$ is the set of observed indices, the goal is to recover $\M$. An optimization version of the problem can be posed as follows:
\begin{equation}
  \label{eq:prob}
(LRMC):\ \ \ \ \ \ \qquad\min_{\X} \frob{\Om{\X-\M}}^2,\ \ s.t.\ \ rank(\X)\leq r,
\end{equation}
where $\Om{\A}$ is defined as:
\begin{equation}
  \label{eq:pom}
\Om{\A}_{ij}=
  \begin{cases}
    A_{ij}, & \text{ if } (i,j)\in \Omega,\\
0, & \text{ otherwise}.
  \end{cases}
\end{equation}
LRMC is by now a well studied problem with applications in several machine learning tasks such as collaborative filtering \cite{BellK2007}, link analysis \cite{GleichL11}, distance embedding \cite{CandesR2007} etc. Motivated by widespread applications, several practical algorithms have been proposed to solve the problem (heuristically) \cite{RechtR2013, HsiehCD12}. 

On the theoretical front, the non-convex rank constraint implies NP-hardness in general \cite{HardtMRW14}. However, under certain (by now) standard assumptions, a few algorithms have been shown to solve the problem efficiently. These approaches can be categorized into the following two broad groups: 

a) The first approach relaxes the rank constraint in \eqref{eq:prob} to a trace norm constraint (sum of singular values of $\X$) and then solves the resulting convex optimization problem \cite{CandesR2007}. \cite{CandesT2009,Recht2009} showed that this approach has a near optimal sample complexity (i.e. the number of observed entries of $\M$) of $\abs{\Omega} = \order{r n \log^2 n}$, where we abbreviate $n = n_1+n_2$. However, current iterative algorithms used to solve the trace-norm constrained optimization problem require $\order{n^2}$ memory and  $\order{n^3}$ time per iteration, which is prohibitive for large-scale applications.

b) The second approach is based on an empirically popular iterative technique called
Alternating Minimization (AltMin) that factorizes $\X=\U \Vtr$ where $\U, \V$ have $r$ columns, and the algorithm alternately optimizes over $\U$ and $\V$ holding the other fixed.  Recently, \cite{Keshavan2012,JainNS2013,Hardt2013,HardtW14} showed convergence of variants of this algorithm. The best known sample complexity results for AltMin are the incomparable bounds $|\Omega|=\order{ r \kappa^8 n {\log \frac{n}{\epsilon}}}$ and $|\Omega|=\order{\; \poly{r} \left(\log \kappa\right) {n \log \frac{n}{\epsilon}}}$ due to \cite{Keshavan2012} and \cite{HardtW14} respectively. Here, $\kappa=\sigma_1(\M)/\sigma_r(\M)$ is the condition number of $\M$ and $\epsilon$ is the desired accuracy. The computational cost of these methods is $\order{|\Omega|r+nr^3}$ per iteration, making these methods very fast as long as the condition number $\kappa$ is not too large.

Of the above two approaches AltMin is known to be the most practical and runs in near linear time. However, its sample complexity as well as computational complexity depend on the condition number of $\M$ which can be arbitrarily large. Moreover, for ``exact'' recovery of $\M$, i.e., with error $\epsilon=0$, the method requires infinitely many samples (or rather to observe the entire matrix). The dependence of sample complexity on the desired accuracy arises due to the use of independent samples in each iteration, which in turn is necessitated by the fact that using the same samples in each iteration leads to complex dependencies among iterates which are hard to analyze. Nevertheless, practitioners have been using AltMin with same samples in each iteration successfully in a wide range of applications.

{\bf Our results}:
In this paper, we address this issue by proposing a new algorithm called Stagewise-SVP (\textbf{\stsvp}) and showing that it solves the matrix completion problem {\em exactly} with a sample complexity $\abs{\Omega}=\order{n r^5 \log^3 n}$, which is independent of both the condition number, and desired accuracy and time complexity per iteration $\order{\abs{\Omega}r^2}$, which is near linear in $n$.

The basic block of our algorithm is a simple projected gradient descent step, first proposed by \cite{JainMD10} in the context of this problem. More precisely, given the $t^{\textrm{th}}$ iterate $\Mto$, \cite{JainMD10} proposed the following update rule, which they call singular value projection (SVP).
\begin{align}
  \label{eq:svp}
  (SVP): \Mto[t+1]=\Pk[r]\left(\Mto+\frac{n_1 n_2}{\abs{\Omega}}\Om{\M - \Mto}\right),
\end{align}
where $\Pk[r]$ is the projection onto the set of rank-$r$ matrices and can be efficiently computed using singular value decomposition (SVD). 
Note that the \svp~step is just a projected gradient descent step where the projection is onto the (non-convex) set of low rank matrices. \cite{JainMD10} showed that despite involving projections onto a non-convex set, \svp~solves the related problem of low-rank matrix sensing, where instead of observing elements of the unknown matrix, we observe dense linear measurements of this matrix. However, their result does not extend to the matrix completion problem and the correctness of \svp~for matrix completion was left as an open question. 

Our preliminary result  resolves this question by showing the correctness of \svp~for  the matrix completion problem, albeit with a sample complexity that depends on the condition number and desired accuracy. We then develop a stage-wise variant of this algorithm, where in the $k^{\textrm{th}}$ stage, we try to recover $\PkM$, there by getting rid of the dependence on the condition number. Finally, in each stage, we use independent samples for $\log n$ iterations, but use same samples for the remaining  iterations, there by eliminating the dependence of sample complexity on  $\epsilon$.

Our analysis relies on two key novel techniques that enable us to understand \svp~style projected gradient methods even though the projection is onto a {\em non-convex} set. First, we consider $\ell_{\infty}$ norm of the error $\Mto-\M$ as our potential function, instead of its spectral norm that most existing analysis of matrix completion use. In general, bounds on the $\ell_{\infty}$ norm are much harder to obtain as projection via SVD is optimal only in the spectral and Frobenius norms. We obtain $\ell_{\infty}$ norm bounds by writing down explicit eigenvector equations for the low rank projection  and using this to control the $\ell_{\infty}$ norm of the error. Second, in order to analyze the \svp~updates with same samples in each iteration, we prove and use a natural extension of the Davis-Kahan theorem. This extension bounds the perturbation in the best rank-$k$ approximation of a matrix (with large enough eigen-gap) due to any additive perturbation; despite this being a very natural extension of the Davis-Kahan theorem, to the best of our knowledge, it has not been considered before. We believe both of the above techniques should be of independent interest.

{\bf Paper Organization}: We first present the problem setup, our main result and an overview of our techniques in the next section. We then present a ``warm-up'' result for the basic \svp~method in Section~\ref{sec:svp}. We then present our main algorithm (\stsvp) and its analysis in Section~\ref{sec:stagewise}. We conclude the discussion in Section~\ref{sec:conc}. The proofs of all the technical lemmas will follow thereafter in the appendix.

{\bf Notation}: We denote matrices with boldface capital letters ($\M$) and vectors with boldface letters ($\x$). $\mathbf{m}_i$ denotes the $i^{\textrm{th}}$ column and $M_{ij}$ denotes the $(i,j)^{\textrm{th}}$ entry respectively of $\M$. SVD and EVD stand for the singular value decomposition and eigenvalue decomposition respectively. $\Pk(\A)$ denotes the projection of $\A$ onto the set of rank-$k$ matrices. That is, if $\A=\U\So \Vtr$ is the SVD of $\A$, then $\Pk(\A)=\U_k \So_k \Vtr_k$ where $\U_k \in \R^{n_1\times k}$ and $\V_k \in \R^{n_2\times k}$ are the $k$ left and right singular vectors respectively of $\A$ corresponding to the $k$ largest singular values $\sigma_1\geq \sigma_2 \geq \dots \geq \sigma_k$. $\norm{\u}_q$ denotes the $\ell_q$ norm of $\u$. We denote the operator norm of $\M$ by $\twonorm{\M}=\max_{\u, \twonorm{\u}=1}\twonorm{\M\u}$. In general, $\twonorm{\mathbf{\alpha}}$ denotes the $\ell_2$ norm of $\mathbf{\alpha}$ if it is a vector and the operator norm of $\mathbf{\alpha}$ if it is a matrix. $\frob{\M}$ denotes the Frobenius norm of $\M$.


\section{Our Results and Techniques}
In this section, we will first describe the problem set up and then present our results as well as the main techniques we use.
\subsection{Problem Setup}\label{sec:prob}
Let $\M$ be an $n_1 \times n_2$ matrix of rank-$r$. Let $\Omega \subseteq [n_1] \times [n_2]$ be a subset of the indices.
Recall that $\Om{\M}$ (as defined in \eqref{eq:pom}) is the projection of $\M$ on to the indices in $\Omega$. Given $\Omega,\,\Om{\M}$ and $r$, the goal is to recover $\M$. The problem is in general ill posed,  so we make the following standard assumptions on $\M$ and $\Omega$ \cite{CandesR2007}.
\begin{assumption}[\textbf{Incoherence}]\label{assump:incoh}
  $\M\in \R^{n_1\times n_2}$ is a rank-$r$, $\mu$-incoherent matrix i.e., $\max_{i} \|\e_i^T\Us\|_2 \leq \frac{\mu\sqrt{r}}{\sqrt{n_1}}$
and $\max_j\|\e_j^T\Vs\|_2 \leq \frac{\mu\sqrt{r}}{\sqrt{n_2}}$, where $\M=\Us\So\Vstr$ is the singular value decomposition of $\M$. 
\end{assumption}
\begin{assumption}[\textbf{Uniform sampling}]\label{assump:unif}
$\Omega$ is generated by sampling each element of $[n_1]\times[n_2]$ independently with probability $p$.
\end{assumption}
The incoherence assumption ensures that the mass of the matrix is well spread out and a small fraction of uniformly random observations give enough information about the matrix. Both of the above assumptions are standard and are used by most of the existing results, for instance \cite{CandesR2007,CandesT2009,KeshavanOM2009,Recht2009,Keshavan2012}. A few exceptions include the works of \cite{MekaJD09,ChenBSW14,BhojanapalliJ14}. 
\subsection{Main Result}\label{sec:results}

The following theorem is the main result of this paper.
\begin{theorem}\label{thm:main}
Suppose $\M$ and $\Omega$ satisfy Assumptions~\ref{assump:incoh} and~\ref{assump:unif} respectively. Also, let \begin{align*}
\E[|\Omega|]\geq C\alpha \mu^4 r^5 n\log^3 n,
\end{align*}
where $\alpha > 1$, $n \defas n_1 + n_2$ and $C>0$ is a global constant.
Then, the output $\whM$ of Algorithm~\ref{algo:newstagewise} satisfies: $\frob{\whM - \M} \leq \epsilon,$ with probability greater than $1- n^{-10-\log \alpha}$. Moreover, the run time of Algorithm~\ref{algo:newstagewise} is $\order{\abs{\Omega}r^2\log(1/\epsilon)}$.
\end{theorem}
Algorithm~\ref{algo:newstagewise} is based on the projected gradient descent update \eqref{eq:svp} and proceeds in $r$ stages where in the $k$-th stage, projections are performed onto the set of rank-$k$ matrices. See Section~\ref{sec:stagewise} for a detailed description and the underlying intuition behing our algorithm.

Table~\ref{tab:comparison} compares our result to that for nuclear norm minimization, which is the only other polynomial time method with finite sample complexity guarantees (i.e. no dependence on the desired accuracy $\epsilon$). Note that \stsvp~runs in time near linear in the ambient dimension of the matrix ($n$), where as nuclear norm minimization runs in time cubic in the ambient dimension. However, the sample complexity of \stsvp~is suboptimal in its dependence on the incoherence parameter $\mu$ and rank $r$. We believe closing this gap between the sample complexity of \stsvp~and that of nuclear norm minimization should be possible and leave it for future work.
\begin{table}[h]
  \begin{center}
    \begin{tabular}{ | c | c | c |}
      \hline
								& Sample complexity & Comp. complexity  \\ \hline
      Nuclear norm minimization \cite{Recht2009}		& $\order{\mu^2 r n \log^2 n}$
					& $\order{n^3 \log \frac{1}{\epsilon}}$ \\ \hline
      \red{\stsvp~(This paper)}	& $\red{ \order{\mu^4 r^5 n \log^3 n}}$
					&  $\red{\order{\mu^4 r^7 n \log^3 n \log(1/\epsilon)}}$ \\ \hline
      \end{tabular}
      \caption{Comparison of our result to that for nuclear norm minimization.}
      \label{tab:comparison}
  \end{center}
\end{table}

\subsection{Overview of Techniques}\label{sec:techniques}
In this section, we briefly present the key ideas and lemmas we use to prove Theorem~\ref{thm:main}. Our proof revolves around analyzing the basic \svp~step \eqref{eq:svp}: $\Mto[t+1]=\Pk\left(\Mto+\frac{1}{p}\Om{\M - \Mto}\right) = \Pk(\M+\wH)$ where $p$ is the sampling probability, $\wH \defas \Mto - \M - \frac{1}{p}\Om{\Mto - \M} = \Em - \frac{1}{p}P_{\Omega}(\Em)$ and $\Em \defas \Mto - \M$ is the error matrix. Hence, $\Mto[t+1]$ is given by a rank-$k$ projection of $\M+\wH$, which is a perturbation of the desired matrix $\M$.

{\bf Bounding the $\ell_\infty$ norm of errors}: As the \svp~update is based on projection onto the set of rank-$k$ matrices, a natural potential function to analyze would be $\twonorm{\Em}$ or $\frob{\Em}$. However, such a potential function requires bounding norms of $\Em - \frac{1}{p}P_{\Omega}(\Em)$ 
which in turn would require us  
 to show that $\Em$ is incoherent.  This is the approach taken by papers on AltMin \cite{Keshavan2012,JainNS2013,Hardt2013}.

In contrast, in this paper, we consider $\infnorm{\Em}$ as the potential function. So the goal is to show that $\infnorm{ \Pk\left(\M+\wH\right) - \M }$ is much smaller than $\infnorm{\Em}$. Unfortunately, standard perturbation results such as the Davis-Kahan theorem provide bounds on spectral, Frobenius or other unitarily invariant norms and do not apply to the $\ell_\infty$ norm.

In order to carry out this argument, we write the singular vectors of $\M+\wH$ as solutions to eigenvector equations and then use these to write $\Mto[t+1]$ explicitly via Taylor series expansion. We use this technique to prove the following more general lemma.

\begin{lemma}\label{lem:error-decay-general}
Suppose $\M\in \R^{n\times n}$ is a symmetric matrix satisfying Assumption~\ref{assump:incoh}. Let $\sk[1] \geq \cdots \geq \sk[r]$ denote its singular values. Let $\Hmat\in \R^{n\times n}$ be a random symmetric matrix such that each $H_{ij}$ is independent with $\expec{H_{ij}}=0$ and $\expec{\abs{H_{ij}}^a} \leq 1/n$ for $2\leq a \leq \log n$. Then, for any $\alpha > 1$ and $|\beta| \leq \frac{{\sk}}{200 \sqrt{\alpha}}$ we have:
\begin{align*}
\infnorm{\M - \Pk\left(\M + \beta \Hmat \right)} \leq \frac{\muzero^2 r^2}{n} \left( \sk[k+1] + 15 |\beta| \sqrt{\alpha} \log n\right),
\end{align*}
with probability greater than $1 - n^{-10-\log \alpha}$.
\end{lemma}
{\bf Proceeding in stages}: If we applied Lemma~\ref{lem:error-decay-general} with $k=r$, we would require $\abs{\beta}$ to be much smaller than $\sigma_r$. Now, $\beta$  can be thought of as $\beta\approx \sqrt{\frac{n}{p}} \infnorm{\Em}$. If we start with $\Mto[0]= 0$, we have $\Em=-\M$, and so $\infnorm{\Em}=\|\M\|_\infty\leq \frac{\sigma_1\mu^2 r}{n}$. To make $\beta\leq \sigma_r$, we would need the sampling probability $p$ to be quadratic in the condition number $\kappa=\sigma_1/\sigma_r$ . In order to overcome this issue, we perform \svp~in $r$ stages with the $k^{th}$ stage performing projections on to the set of rank-$k$ matrices while maintaining the invariant that at the end of $(k-1)^{\textrm{th}}$ stage, $\infnorm{\Em}=O(\sigma_k/n)$. This lets us choose a $p$ independent of $\kappa$ while still ensuring $\beta \approx \sqrt{\frac{n}{p}} \infnorm{\Em} \leq \sk$. Lemma~\ref{lem:error-decay-general} tells us that at the end of the $k^{\textrm{th}}$ stage, the error $\infnorm{\Em}$ is $\order{\frac{\sk[k+1]}{n}}$, there by establishing the invariant for the $(k+1)^{\textrm{th}}$ stage.
%

{\bf Using same samples}: In order to reduce the error from $\order{\frac{\sk}{n}}$ to $\order{\frac{\sk[k+1]}{n}}$, the $k^{\textrm{th}}$ stage would require $\order{\log \frac{\sk}{\sk[k+1]}}$ iterations. Since Lemma~\ref{lem:error-decay-general} requires the elements of $\Hmat$ to be independent, in order to apply it, we need to use fresh samples in each iteration. This means that the sample complexity increases with $\frac{\sk}{\sk[k+1]}$, or the desired accuracy $\epsilon$ if $\epsilon < \sk[k+1]$. This problem is faced by all the existing analysis for iterative algorithms for matrix completion \cite{Keshavan2012,JainNS2013,Hardt2013,HardtW14}. We tackle this issue by observing that when $\M$ is ill conditioned and $\frob{\Em}$ is very small, we can show a decay in $\frob{\Em}$ using the same samples for \svp~iterations:\\
\begin{lemma}\label{lem:error-decay-samesamples}
Let $\M$ and $\Omega$ be as in Theorem~\ref{thm:main} with $\M$ being a symmetric matrix. Further, let $\M$ be ill conditioned in the sense that $\frob{\M - P_k(\M)} < \frac{\sk}{n^3}$, where $\sk[1]\geq \cdots \geq \sk[r]$ are the singular values of $\M$. Then, the following holds for all rank-$k$ $\X$ s.t. $\frob{\X - \Pk(\M)} < \frac{\sk}{n^3}$ (w.p. $\geq 1-n^{-10-\alpha}$):
\begin{align*}
\frob{\Xplus - P_k(\M)} \leq \frac{1}{10} \frob{\X - \Pk(\M)} + \frac{1}{{p}} \frob{\M - \Pk(\M)},
\end{align*}
where $\Xplus \defas \Pk\left(\X - \frac{1}{p} \Om{ \X-\M }\right)$ denotes the rank-$k$ \svp~update of $\X$ and $p=\E[|\Omega|]/n^2=\frac{C\alpha\mu^4r^5\log^3 n}{n}$ is the sampling probability.
\end{lemma}

The following lemma plays a crucial role in proving Lemma~\ref{lem:error-decay-samesamples}. It is a natural extension of the Davis-Kahan theorem for singular vector subspace perturbation.
\begin{lemma}\label{lem:daviskahan-approx}
Suppose $\A$ is a matrix such that $\sigma_{k+1}(\A) \leq \frac{1}{4}\sigma_k(\A)$. Then, for any matrix $\Em$ such that $\frob{\Em} < \frac{1}{4}\sigma_k(\A)$, we have:
\begin{align*}
\frob{P_k\left(\A+\Em\right) - P_k\left(\A\right)} \leq c\left( \sqrt{k}\twonorm{\Em} + \frob{\Em} \right),
\end{align*}
for some absolute constant $c$.
\end{lemma}
In contrast to the Davis-Kahan theorem, which establishes a bound on the perturbation of the space of singular vectors, Lemma~\ref{lem:daviskahan-approx} establishes a bound on the perturbation of the best rank-$k$ approximation of a matrix $\A$ with good eigen gap, under small perturbations. This is a very natural quantity while considering perturbations of low rank approximations, and we believe it may find applications in other scenarios as well. A final remark regarding Lemma~\ref{lem:daviskahan-approx}: we suspect it might be possible to tighten the right hand side of the result to $c \min\left(\sqrt{k} \twonorm{\Em}, \frob{\Em}\right)$, but have not been able to prove it.
\newpage
\section{Singular Value Projection}\label{sec:svp}
\begin{algorithm}[t]
 \caption{SVP for matrix completion}
 \begin{algorithmic}[1]
   \STATE {\bf Input}: $\Omega, P_{\Omega}(\M), r, \epsilon$
   \STATE $T \leftarrow \log \frac{(n_1+n_2) \infnorm{\M}}{\epsilon}$
   \STATE Partition $\Omega$ randomly into $T$ subsets $\set{\Omega_t: t\in[T]}$
   \STATE $\Mto \leftarrow 0$
   \FOR{$t \leftarrow 1,\cdots,T$}
   \STATE $\Mto \leftarrow P_r\left(\Mto[t-1] - \frac{n_1n_2}{\abs{\Omega_t}}P_{\Omega_t}\left(\Mto[t-1]-\M\right)\right)$
   \ENDFOR
   \STATE {\bf Output}: $\Mto[T]$
 \end{algorithmic}
 \label{algo:svp}
\end{algorithm}
Before we go on to prove Theorem~\ref{thm:main}, in this section we will analyze the basic \svp~algorithm (Algorithm~\ref{algo:svp}), bounding its sample complexity and thereby resolving a question posed by Jain et al. \cite{JainMD10}. This analysis also serves as a warm-up exercise for our main result and brings out the key ideas in analyzing the $\ell_\infty$ norm potential function while also highlighting some issues with Algorithm~\ref{algo:svp} that we will fix later on.

As is clear from the pseudocode in Algorithm~\ref{algo:svp}, \svp~is a simple projected gradient descent method for solving the matrix completion problem.
Note that Algorithm~\ref{algo:svp} first splits the set $\Omega$ into $T$ random subsets and updates iterate $\Mto$ using $\Omega_t$. This step is critical for analysis as it ensures that $\Omega_t$ is independent of $\Mto[t-1]$, allowing for the use of standard tail bounds. The following theorem is our main result for Algorithm~\ref{algo:svp}:
\begin{theorem}\label{thm:cond_svp}
Suppose $\M$ and $\Omega$ satisfy Assumptions~\ref{assump:incoh} and~\ref{assump:unif} respectively with
\begin{align*}
\expec{\abs{\Omega}} \geq C \alpha \mu^4 \kappa^2 r^5 n \left(\log^2 n\right) T,
\end{align*}
where $n=n_1+n_2, \alpha > 1, \kappa=\left(\frac{\sigma_1}{\sigma_r}\right)$ with $\sk[1] \geq \cdots \geq \sk[r]$ denoting the singular values of $\M$, $T=\log \frac{100 \mu^2 r \twonorm{\M}}{\epsilon}$ and $C>0$ is a large enough global constant. Then, the output of Algorithm~\ref{algo:svp} satisfies (w.p. $\geq 1-T{\min(n_1,n_2)^{-10-\log \alpha}}$):
$ \frob{\Mto[T]-\M} \leq \epsilon$
\end{theorem}
\begin{proof}
Using a standard dilation argument (Lemma~\ref{lem:symm}), it suffices to prove the result for symmetric matrices.
Let $p = \frac{\expec{\abs{\Omega_t}}}{n^2} = \frac{\expec{\abs{\Omega}}}{n^2T}$ be the probability of sampling in each iteration. Now, let $\Em=\Mto[t-1]-\M$ and $\wH=\Em-\frac{1}{p}P_{\Omega_t}(\Em)$. Then, the SVP update (line 6 of Algorithm~\ref{algo:svp}) is given by: $\Mto=\Pk[r](\M+\wH)$.
Since $\Omega_t$ is sampled uniformly at random, it is easy to check that $\mathbb{E}[\widehat{H}_{ij}]=0$ and $\expec{\abs{\widehat{H}_{ij}}^s} \leq \beta^s/n$ where $\beta=\frac{2\sqrt{n}\infnorm{E}}{\sqrt{p}}\leq \frac{2 \mu^2 r\sk[1]}{\sqrt{n p}}$ (Lemma~\ref{lem:sat-defn}). By our choice of $p$, we have $\beta<\frac{\sk[r]}{200\sqrt{\alpha}}$. Applying Lemma~\ref{lem:error-decay-general} with $k=r$, we have 
$\infnorm{\Mto-\M} \leq \frac{15\mu^2 r^2}{n} \beta \sqrt{\alpha} \log n \leq (1/30 C) \infnorm{\Em} = \frac{1}{2}\infnorm{\Mto[t-1]-\M}$, where the last inequality is obtained by selecting $C$ large enough. The theorem is immediate from this error decay in each step.
\end{proof}
\newpage

\section{Stagewise-SVP}\label{sec:stagewise}

\begin{algorithm}[t]
  \caption{Stagewise SVP (\stsvp) for matrix completion}
  \label{algo:newstagewise}
  \begin{algorithmic}[1]
    \STATE {\bf Input}: $\Omega, P_{\Omega}(\M), \epsilon, r$
    \STATE $T \leftarrow \log\frac{100 \muzero^2 r \twonorm{\M}}{\epsilon}$
    \STATE Partition $\Omega$ into $r \log n$ subsets $\set{\Omega_{k,t}: k\in[r],t\in[\log n]}$ uniformly at random
    \STATE $k \leftarrow 1$,  $\Mt[0] \leftarrow 0$
    \FOR{$k \leftarrow 1,\cdots,r$}
    \STATE {\em /* Stage-$k$ */}
    \FOR{$t = 1,\cdots,\log n$}
    \STATE $\Mt \leftarrow \pgd(\Mt[t-1], P_{\Omega_{k,t}}(M), \Omega_{k,t}, k)${\em /* SVP Step with re-sampling*/}
\hspace{1em}\rlap{\smash{$\left.\begin{array}{@{}c@{}}\\{}\\{}\end{array} \color{red} \right\}%
          \color{red}\begin{tabular}{l}{Step I}\end{tabular}$}}
    \ENDFOR
	\IF{$\sigma_{k+1}\left(\gd\left(\Mt[\log n], P_{\Omega_{k,\log n}}(\M), \Omega_{k,\log n}\right) \right) > \frac{\sigma_k\left(\Mt[\log n]\right)}{n^2}$}
	\STATE $\Xkpt[0] \leftarrow \Mt[T]$ {\em /* Initialize for next stage and continue*/}
\hspace{6.7em}\rlap{\smash{$\left.\begin{array}{@{}c@{}}\\{}\\{}\end{array} \color{red} \right\}%
          \color{red} \begin{tabular}{l} {Step II} \end{tabular}$}}
        \STATE {\bf continue}
	\ENDIF
	\FOR{$t = \log n+1,\cdots,\log n+T$}
	\STATE $\Mt[t] \leftarrow \pgd(\Mt[t-1],P_{\Omega}(\M), \Omega, k)$ \hspace*{10pt}{\em /* SVP Step without re-sampling */}
\rlap{\smash{$\left.\begin{array}{@{}c@{}}\\{}\\{}\end{array} \color{red} \right\}%
          \color{red} \begin{tabular}{l} {Step III} \end{tabular}$}}
	\ENDFOR
	\FOR{$t = \log n+T+1,\cdots,\log n+T+\log n$}
        \STATE $\Mt[t] \leftarrow \pgd(\Mt[t-1], P_{\Omega_{k,t}}(M), \Omega_{k,t}, k)$\hspace*{10pt}{\em /* SVP Step with re-sampling */}
\rlap{\smash{$\left.\begin{array}{@{}c@{}}\\{}\\{}\end{array} \color{red} \right\}%
          \color{red} \begin{tabular}{l} {Step IV} \end{tabular}$}}
	\ENDFOR
        \STATE $\Xkpt[0] \leftarrow \Mt[t]$ {\em /* Initialization for next stage */}
        \STATE {\bf Output}: $\Mt[t]$ if $\sigma_{k+1}\left(GD(\Mt[t- 1],P_{\Omega_{k,t}}(\M), \Omega_{k, t})\right) < \frac{\epsilon}{10 \muzero^2 r}$
    \ENDFOR
  \end{algorithmic}
\end{algorithm}

\floatname{algorithm}{Sub-routine}

\begin{algorithm}[t!]\caption{Projected Gradient Descent (\pgd)}
  \begin{algorithmic}[1]
    \STATE {\bf Input}: $X\in \R^{n_1\times n_2}, P_{\Omega}(\M), \Omega, k$
    \STATE {\bf Output}: $X_{next}\leftarrow P_k(X-\frac{n_1n_2}{|\Omega|}P_{\Omega}(X-M))$
  \end{algorithmic}
\end{algorithm}

\begin{algorithm}[t!]\caption{Gradient Descent (\gd)}
  \begin{algorithmic}[1]
    \STATE {\bf Input}: $X\in \R^{n_1\times n_2}, P_{\Omega}(\M), \Omega$
    \STATE {\bf Output}: $X_{next}\leftarrow X-\frac{n_1n_2}{|\Omega|}P_{\Omega}(X-M)$
  \end{algorithmic}
\end{algorithm}
\floatname{algorithm}{Algorithm}
Theorem~\ref{thm:cond_svp} is suboptimal in its sample complexity dependence on the rank, condition number and desired accuracy. In this section, we will fix two of these issues -- the dependence on condition number and desired accuracy -- by designing a stagewise version of Algorithm~\ref{algo:svp} and proving Theorem~\ref{thm:main}.

Our algorithm, \stsvp~(pseudocode presented in Algorithm~\ref{algo:newstagewise}) runs in $r$ stages, where in the $k^{\textrm{th}}$ stage, the projection is onto the set of {\em rank-$k$ matrices}. In each stage, the goal is to obtain an approximation of $\M$ up to an error of $\sk[k+1]$. In order to do this, we use the basic \svp~updates, but in a very specific way, 
so as to avoid the dependence on condition number and desired accuracy. 
\begin{itemize}
\item
\textbf{(Step I) Apply SVP update with fresh samples for $\log n$ iterations}: Run $\log n$ steps of SVP update \eqref{eq:svp}, with fresh samples in each iteration. Using fresh samples allows us to use Lemma~\ref{lem:error-decay-general} ensuring that the $\ell_{\infty}$ norm of the error between our estimate and $\M$ decays to $\infnorm{ \Mt[\log n]-\M} = \order{\frac{1}{n}\left(\sk[k+1] + \frac{\sk}{n^3}\right)}$.
\item
\textbf{(Step II) Determine if $\sk[k+1]>\frac{\sk}{n^3}$}: Note that we can determine this, by using the $(k+1)^{\textrm{th}}$ singular value of the matrix obtained after the gradient step, i.e., $\sigma_{k+1}(\Mt[\log n]-\frac{1}{p}P_{\Omega_{k, \log n}}(\Mt[\log n]-\M))$. If true, the error $ \infnorm{ \Mt[\log n]-\M} = \order{\frac{\sk[k+1]}{n}}$, and so the algorithm proceeds to the $(k+1)^{\textrm{th}}$ stage.
\item
\textbf{(Step III) If not (i.e., $\sk[k+1] \leq \frac{\sk}{n^3}$), apply SVP update for $T = \log \frac{1}{\epsilon}$ iterations with same samples}: If $\sk[k+1] \leq \frac{\sk}{n^3}$, we can use Lemma~\ref{lem:error-decay-samesamples} to conclude that after $\log \frac{1}{\epsilon}$ iterations, the Frobenius norm of error is $\frob{ \Mt[\log n+T] - \M}=\order{{n} \sk[k+1] + \epsilon}$.
\item
\textbf{(Step IV) Apply \svp~update with fresh samples for $\log n$ iterations}: To set up the invariant $\infnorm{ \X_{k+1,0} - \M} = \order{\sigma_{k+1}/n}$ for the next stage, we wish to convert our Frobenius norm bound $\frob{\Mt[\log n+T]-\M} = \order{{n} \sk[k+1]}$ to an $\ell_\infty$ bound $\infnorm{\Mt[2 \log n+T]-\M} = \order{\frac{\sk[k+1]}{n}}$. Since $\sk[k+1] < \frac{\sk}{n^3}$, we can bound the initial Frobenius error by $\order{\frac{1}{n}\left(\left(\frac{1}{2}\right)^{\hT} \sk + \sk[k+1] \right)}$ for some ${\hT} = \order{\log \frac{\sk}{n^2\sk[k+1]}}$. As in  Step I, after $\log n$ \svp~updates with fresh samples, Lemma~\ref{lem:error-decay-general} lets us conclude that $\infnorm{\Mt[2 \log n+T]-\M} = \order{\frac{\sk[k+1]}{n}}$, setting up the invariant for the next stage.
\end{itemize}

\subsection{Analysis of \stsvp~(Proof of Theorem~\ref{thm:main})}\label{sec:st_svp_analysis}
We will now present a proof of Theorem~\ref{thm:main}.
\begin{proof}[Proof of Theorem~\ref{thm:main}]
Just as in Theorem~\ref{thm:cond_svp}, it suffices to prove the result for when $\M$ is symmetric.
For every stage, we will establish the following invariant:
\begin{align}\label{eq:st_invar}
	\infnorm{\Mt[0] - \M} < \frac{4\mu^2 r^2}{n}  {\sk[k+1]}.
\end{align}
We will use induction. \eqref{eq:st_invar} clearly holds for the base case $k=1$. Now, suppose \eqref{eq:st_invar} holds for the $k^{\textrm{th}}$ stage, we will prove that it holds for the $(k+1)^{\textrm{th}}$ stage. The analysis follows the four step outline in the previous section: 

\textbf{Step I}: Here, we will show that for every iteration $t$, we have:
\begin{align}
	\infnorm{\Mt - \M} < \frac{4\mu^2 r^2}{n} \gkt,\ \text{where}\  \gkt \defas {\sk[k+1]} + \left(\frac{1}{2}\right)^{t-1} {\sk}.\label{eqn:gkt-update}
\end{align}
\eqref{eqn:gkt-update} holds for $t=0$ by our induction hypothesis \eqref{eq:st_invar} for the $k$-th stage. Supposing it true for iteration $t$, we will show it for iteration $t+1$. The $(t+1)^{\textrm{th}}$ iterate is given by: 
\begin{equation}\label{eqn:update} 
\Mt[t+1] = P_k\left(\M+\beta \Hmat\right),\text{ where }\Hmat=\frac{1}{\beta}\left(\Em-\frac{1}{p}P_{\Omega_{k, t}}(\Em)\right), \Em=\Mt-\M,\ 
\end{equation}
$p = \frac{\expec{\abs{\Omega_{k,t}}}}{n^2}=\frac{C\alpha\mu^4r^4\log^2 n}{n}$, and $\beta=\frac{2\sqrt{n}\infnorm{\Em}}{\sqrt{p}}\leq \frac{8\mu^2 r^2\gkt}{\sqrt{n\cdot p}}$. Our hypothesis on the sample size tells us that $\beta\leq \abs{ \sigma_k} / (200\sqrt{\alpha})$ and Lemma~\ref{lem:sat-defn} tells us that $\Hmat$ satisfies the hypothesis of Lemma~\ref{lem:error-decay-general}. So we have: 
\begin{align*}
\infnorm{\Mt[t+1] - \M} < \frac{\mu^2 r^2 }{n} \left({\sk[k+1]} + 15 \beta \sqrt{\alpha} \log n\right)
< \frac{\mu^2 r^2}{n} \left({\sk[k+1]} + \frac{1}{9} \gkt \right) \leq \frac{10 \mu^2 r^2}{9n} \gkt[t+1].
\end{align*}
This proves~\eqref{eqn:gkt-update}. Hence, after $\log n$ steps, we have:
\begin{align}
	\infnorm{\Mt[\log n] - \M} < \frac{10\mu^2 r^2}{9n}\left(\frac{{\sk}}{n^3} + {\sk[k+1]} \right).\label{eq:st_ph1}
\end{align}

\textbf{Step II}: Let $\G \defas \Mt[\log n] - \frac{1}{p}P_{\Omega_{k,\log n}}\left(\Mt[\log n] - \M\right)=\M+\beta \Hmat$ be the gradient update with notation as above.
A standard perturbation argument (Lemmas~\ref{lem:cond_spec1} and~\ref{lem:weyl-perturbation}) tells us that:
\begin{align*}
\twonorm{\G - \M} < 3 \beta \sqrt{\alpha} \leq \frac{24\mu^2 r^2 \gkt[\log n]}{\sqrt{np}}
< \frac{1}{100} \left(\frac{{\sk}}{n^3} +{\sk[k+1]} \right).
\end{align*}
So if $\sigma_{k+1}(\G) > \frac{\sigma_k(\G)}{n^3}$, then we have $\sk[k+1] > \frac{9\sk}{10n^3}$.
Since we move on to the next stage with $\Xkpt[0] = \Mt[\log n]$,~\eqref{eq:st_ph1} tells us that:
\begin{align*}
\infnorm{\Xkpt[0] - \M} = \infnorm{\Mt[\log n] - \M}
&\leq \frac{10 \mu^2 r^2}{9 n} \left(\frac{\sk}{n^3} + \sk[k+1]\right)
\leq \frac{2 \mu^2 r^2 }{n} \left( 2 \sk[k+1] \right),
\end{align*}
showing the invariant for the $(k+1)^{\textrm{th}}$ stage.

\textbf{Step III}:
On the other hand, if $\sigma_{k+1}(\G) \leq \frac{\sigma_k(\G)}{n^3}$, then Lemmas~\ref{lem:weyl-perturbation} and~\ref{lem:weyl-perturbation} tell us that $\sk[k+1] \leq \frac{11 \sk}{10 n^3}$. So, using Lemma~\ref{lem:error-decay-samesamples} with $T = \log \frac{1}{\epsilon}$ iterations, we obtain:
\begin{align}
\frob{\Mt[T+\log n] - \PkM } \leq \max\left(\epsilon, \frac{2}{{p}} \frob{\M - \PkM} \right).\label{eq:st_ph3a}
\end{align}
If $ \epsilon > \frac{2}{{p}} \frob{\M - \PkM}$, then we have:
\begin{align*}
\frob{\Mt[T+\log n] - \M} \leq \frob{\Mt[T+\log n] - \PkM} + \frob{\M - \PkM} \leq 2 \epsilon.
\end{align*}
On the other hand, if $\epsilon \leq \frac{2}{{p}} \frob{\M - \PkM}$, then we have:
\begin{align}
\infnorm{\Mt[T+\log n] - \M } &\leq
\frob{\Mt[T+\log n] - \PkM } + \infnorm{\M - \PkM} \nonumber\\
&\leq \frac{2}{{p}} \frob{\M - \PkM} + \frac{\mu^2 r^2 \sk[k+1]}{n}
\leq \left(2\sqrt{r}n+\frac{\mu^2 r^2}{n}\right) \sk[k+1] \nonumber\\
&\leq \frac{2 \mu^2 r^2}{n}\left(\left(\frac{1}{2}\right)^{\log \frac{\sk}{n^2 \sk[k+1]}} \sk + \sk[k+1] \right).\label{eq:st_ph3}
\end{align}

\textbf{Step IV}: Using \eqref{eq:st_ph3} and ``fresh samples'' analysis as in Step I (in particular~\eqref{eqn:gkt-update}), we have: 
\begin{align*}
\infnorm{\Mt[T+2 \log n] - \M } \leq
\frac{10 \mu^2 r^2}{9n}\left(\left(\frac{1}{2}\right)^{\log \frac{\sk}{\sk[k+1]}} \sk + \sk[k+1] \right)
\leq \frac{2 \mu^2 r^2}{n}\left(2 \sk[k+1]\right),
\end{align*}
which establishes the invariant for the $(k+1)^{\textrm{th}}$ stage.

Combining the invariant \eqref{eq:st_invar} with the exit condition as in Step III, we have: $\|\widehat{\M}-\M\|_F\leq \epsilon$ where $\widehat{\M}$ is the output of the algorithm. As there are $r$ stages, and in each stage, we need $2\log n$ sets of samples of size $O(pn^2)$. Hence, the total samplexity is $|\Omega|=\order{\alpha \mu^4 r^5n\log^3 n}$. Similarly, total computation complexity is $\order{\alpha \mu^4 r^7n\log^3 n\log(\|M\|_F/\epsilon)}$.
\end{proof}


\section{Discussion and Conclusions}\label{sec:conc}
In this paper, we proposed a fast projected gradient descent based algorithm for solving the matrix completion problem. The algorithm runs in time $\order{nr^7\log^3 n \log 1/\epsilon}$, with a sample complexity of $\order{nr^5 \log^3 n}$.
To the best our knowledge, this is the first near linear time algorithm for exact matrix completion with sample complexity independent of $\epsilon$ and condition number of $\M$.

The first key idea behind our result is to use the $\ell_\infty$ norm as a potential function which entails bounding all the terms of an explicit Taylor series expansion.
The second key idea is an extension of the Davis-Kahan theorem, that provides  perturbation bound  for the best rank-$k$ approximation of a matrix with good eigen-gap. We believe both these techniques may find applications in other contexts.

Design an efficient algorithm with information-theoretic optimal sample complexity  $\abs{\Omega}=\order{n r \log n}$ is still open; our result is suboptimal by a factor of ${r^4 \log^2 n}$ and nuclear norm approach is suboptimal by a factor of $\log n$. Another interesting direction in this area is to design optimal algorithms that can handle sampling distributions that are  widely observed in practice, such as the power law distribution\cite{MekaJD09}.

 \clearpage
 \bibliographystyle{alpha}

\clearpage
\appendix

\section{Preliminaries and Notations for Proofs}
The following lemma shows that wlog we can assume $\M$ to be a symmetric matrix. A similar result is given in Section D of \cite{Hardt2013}. 
\begin{lemma}
  \label{lem:symm}
Let $\M\in \R^{n_1\times n_2}$ and $\Omega\subseteq [n_1]\times [n_2]$ satisfy Assumption~\ref{assump:incoh} and \ref{assump:unif}, respectively. Then, there exists a symmetric $\widetilde{\M}\in \R^{n\times n}$, $n=n_1+n_2$, s.t. $\widetilde{\M}$ is of rank-$2r$, incoherence of $\widetilde{\M}$ is twice the incoherence of $\M$. Moreover, there exists $|\widetilde{\Omega}|\subseteq [n]\times [n]$ that satisfy Assumption \ref{assump:unif}, $P_{\widetilde{\Omega}}(\widetilde{\M})$ is efficiently computable, and the output of a SVP update \eqref{eq:svp} with $P_{\Omega}(\M)$ can also be obtained by the SVP update of $P_{\widetilde{\Omega}}(\widetilde{\M})$. 
\end{lemma}
\begin{proof}[Proof of Lemma~\ref{lem:symm}]
Define the following symmetric matrix from $\M$ using a dilation technique:
\begin{align*}
  \widetilde{\M}=\left[\begin{matrix}0&\M\\ \trans{\M}&0\end{matrix}\right].
\end{align*}
Note that the rank of $\widetilde{\M}$ is $2\cdot r$ and the incoherence of $\widetilde{\M}$ is bounded by $(n_1+n_2)/n_2 \mu$ (assume $n_1\leq n_2$). Note that if $n_2>  n_1$, then we can split the columns of $\M$ in blocks of size $n_1$ and apply the argument separately to each block. 

Now, we can split $\Omega$ to generate samples from $\M$ and $\M^T$, and then augment redundant samples from the $0$ part above to obtain $\widetilde{\Omega}=[n]\times [n]$. 

Moreover, if we run the SVP update \eqref{eq:svp} with input $\widetilde{\M}$, $\widetilde{\X}$ and $\widetilde{\Omega}$, an easy calculation shows that the iterates satisfy:
\begin{align*}
  \widetilde{\X}_{+}=\left[\begin{matrix}0& \X_{+} \\ \trans{\X_{+}} &0\end{matrix}\right],
\end{align*}
where $\X_+$ is the output of \eqref{eq:svp} with input $\M$, $\X$, and $\Omega$. That is, a convergence result for $\widetilde{\X}_+$ would imply a convergence result for $\X_+$ as well. 
\end{proof}
{\bf For the remaining sections}, we assume (wlog) that $\M\in \R^{n\times n}$ is symmetric and $\M=\Us\Sigma\Ustr$ is the eigenvalue decomposition (EVD) of $\M$. Also, unless specified, $\sigma_i$ denotes the $i$-th eigenvalue of $\M$. 
\section{Proof of Lemma~\ref{lem:error-decay-general}}
Recall that we assume (wlog) that $\M\in \R^{n\times n}$ is symmetric and $\M=\Us\So\Ustr$ is the eigenvalue decomposition (EVD) of $\M$. Also, the goal is to bound $\|\Xplus-\M\|_\infty$, where $\Xplus=P_k(\M+\beta \Hmat)$ and $\Hmat$ is such that it satisfies the following definition: 
\begin{definition}\label{defn:moment-condns}
  $\Hmat$ is a symmetric matrix with each of its elements drawn independently, satisfying the following moment conditions:
\begin{center}
\begin{tabular}{ccc}
  $\expec{h_{ij}} = 0,$ &$|h_{ij}|<1$, & $\expec{\abs{h_{ij}}^k} \leq \frac{1}{n}$,
\end{tabular}
\end{center}
for $i,j \in [n]$ and $2 \leq k \leq 2 \log n$.
\end{definition}

That is, we wish to understand $\|\Xplus-\M\|_\infty$ under perturbation $\Hmat$. To this end, we first present a few lemmas that analyze how $\Hmat$ is obtained in the context of our \stsvp~ algorithm and also bounds certain key quantities related to $\Hmat$. We then present a few technical lemmas that are helpful for our proof of Lemma~\ref{lem:error-decay-general}. The detailed proof of the lemma is given in Section~\ref{sec:mainproof-error-decay-general}. See Section~\ref{sec:techproof-error-decay-general} for proofs of the technical lemmas. 
\subsection{Results for $\Hmat$}\label{sec:Hmat}
Recall that the SVP update \eqref{eq:svp} is given by: $\Xplus=P_k(\X-\frac{1}{p}P_{\Omega}(\X-\M))=P_k(\M+\Hmat)$ where $\Hmat=\mE-\frac{1}{p}P_{\Omega}(\mE)$ and $\mE=\X-\M$. Our first lemma shows that matrices of the form $\mE - \frac{1}{p}P_{\Omega}(\mE)$, scaled appropriately, satisfy Definition~\ref{defn:moment-condns}, i.e., satisfies the assumption of Lemma~\ref{lem:error-decay-general}.
\begin{lemma}\label{lem:sat-defn}
Let $\A$ be a symmetric $n \times n$ matrix. Suppose $\Omega \subseteq [n]\times[n]$ is obtained by sampling each element
with probability $p \in \left[\frac{1}{4n},0.5\right]$. Then the matrix
\begin{align*}
\mB \defas \frac{\sqrt{p}}{2 \sqrt{n} \infnorm{\A}} \left(\A - \frac{1}{p}\PoA\right)
\end{align*}
satisfies Definition~\ref{defn:moment-condns}.
\end{lemma}

We now present a critical lemma for our proof which bounds $\|H^au\|_\infty$ for $2\leq a\leq \log n$. Note that the entries of $H^a$ can be dependent on each other, hence we cannot directly apply standard tail bounds. Our proof follows along very similar lines to Lemma~6.5 of \cite{ErdosKYY2012}; see Appendix~\ref{app:Hpbound} for a detailed proof. 
\begin{lemma}\label{lem:Htp-spectralnorm_app}
Suppose $\wH$ satisfies Definition~\ref{defn:moment-condns}.
Fix $1 \leq a \leq \log n$. Let $\er$ denote the $r^{\textrm{th}}$ standard basis vector.
Then, for any fixed vector $\u$, we have:
\begin{align*}
  \abs{\ip{\er}{\wH^a \u}} \leq \left(c \log n\right)^{a} \infnorm{\u} \; \forall \; r \in [n],
\end{align*}
with probability greater than $1 - n^{1 - 2 \log \frac{c}{4}}$.
\end{lemma}
Next, we bound $\|H\|_2$ using matrix Bernstein inequality by \cite{tropp}; see Appendix~\ref{sec:techproof-error-decay-general} for a proof. 
\begin{lemma}\label{lem:cond_spec1}
Suppose $\Hmat$ satisfies Definition~\ref{defn:moment-condns}. 
Then, w.p. $\geq 1-1/n^{10+\log \alpha}$, we have: $\twonorm{\Hmat}\leq 3\sqrt{\alpha}.$ 
\end{lemma}


\subsection{Technical Lemmas useful for Proof of Lemma~\ref{lem:error-decay-general}}\label{sec:tech-error-decay-general}
In this section, we present the technical lemmas used by our proof of Lemma~\ref{lem:error-decay-general}. 

First, we present the well known Weyl's perturbation inequality \cite{bhatia}: 
\begin{lemma}\label{lem:weyl-perturbation}
Suppose $\mB = \A + \Nmat$. Let $\lami[1],\cdots,\lami[n]$ and $\sk[1],\cdots,\sk[n]$ be the eigenvalues of
$\mB$ and $\A$ respectively. Then we have:
\begin{align*}
  \abs{\lami - \sk[i]} \leq \twonorm{\Nmat} \; \forall \; i \in [n].
\end{align*}
\end{lemma}

The below given lemma bounds the $\linf$ norm of an appropriate incoherent matrix using its $\ltwo$ norm. 
\begin{lemma}\label{lem:linf-l2}
  Suppose $\M$ is a symmetric matrix with size $n$ and satisfying Assumption~\ref{assump:incoh}.
  For any symmetric matrix $\mB \in \R^{n \times n}$, we have:
\begin{align*}
  \infnorm{\M {\mB} \M - \M} \leq \frac{\mu^2 r}{n} \twonorm{\M \mB \M - \M}.
\end{align*}
\end{lemma}

Next, we present a natural perturbation lemma that bounds the spectral norm distance of $A$ to $AB^{-1}A$ where $B=P_k(A+E)$ and $E$ is a perturbation to $A$. 
\begin{lemma}\label{lem:cond_spec}
Let $\A \in \R^{n\times n}$ be a symmetric matrix with eigenvalues $\beta_1, \cdots, \beta_n$,
where $|\beta_1|\geq \cdots \geq |\beta_n|$. Let $\W=\A+\mE$ be a perturbation
of $\A$, where $\mE$ is a symmetric matrix with $\twonorm{\mE} < \frac{\abs{\beta_k}}{2}$. Also, let $\Pk(\W)=\U \Lo \trans{\U}$
be the eigenvalue decomposition of the best rank-$k$ approximation of $\W$.
Then, $\Lo^{-1}$ exists. Furthermore, we have:
\begin{align*}
\twonorm{\A-\A \U \Lo^{-1} \trans{\U} \A} &\leq \abs{\beta_{k+1}} + 5\twonorm{\mE}, \mbox{ and }\\
\twonorm{\A \U \Lo^{-a}\trans{\U}\A} &\leq 4\left(\frac{\abs{\beta_k}}{2}\right)^{-a+2} \; \forall \; a \geq 2.
\end{align*}
\end{lemma}

\subsection{Detailed Proof of Lemma~\ref{lem:error-decay-general}}\label{sec:mainproof-error-decay-general}
We are now ready to present a proof of Lemma~\ref{lem:error-decay-general}. Recall that $\Xplus=P_k(\M+\beta \Hmat)$, hence, 
\begin{equation}(\M + \beta \Hmat)\u_i = \lambda_i \u_i,\ \forall 1\leq i\leq k,\label{eq:svp_eig_update}\end{equation}
where $(\u_i,\lambda_i)$ is the $i^{\textrm{th}} \, (i \leq k)$ top eigenvector-eigenvalue pair (in terms of magnitude).

Now, as $\Hmat$ satisfies conditions of Definition~\ref{defn:moment-condns}, we can apply Lemma~\ref{lem:cond_spec1} to obtain: 
\begin{align}\label{eqn:perturb-main-new-app}
 \abs{\beta} \twonorm{\Hmat} \leq \abs{\beta} \cdot 3 \sqrt{\alpha} \leq \frac{\abs{\sk}}{5}.
\end{align}

Using Lemma~\ref{lem:weyl-perturbation} and~\eqref{eqn:perturb-main-new-app}, we have: 
\begin{align}\label{eq:tmp_pert}
\abs{\lambda_i} \geq \abs{\so_i} - \abs{\beta} \twonorm{\Hmat} \geq \frac{4 \abs{\so_k}}{5} \; \forall \; i \in [k].
\end{align}

Using \eqref{eq:svp_eig_update}, we have: $\left(\eye - \frac{\beta}{\lambda_i} \Hmat \right) \u_i = \frac{1}{\lambda_i} \M \u_i$. Moreover, using \eqref{eq:tmp_pert}, $\eye - \frac{\beta}{\lambda_i} \Hmat$ is invertible. Hence, using Taylor series expansion, we have: 
\begin{align*}
  \u_i = \frac{1}{\lambda_i}\left(\eye + \frac{\beta}{\lambda_i} \Hmat + \frac{\beta^2}{\lambda_i^2} \left(\Hmat\right)^2 + \cdots \right) \M  \u_i.
\end{align*}
Letting $\U \Lo \Utr$ denote the eigenvalue decomposition (EVD) of $\Xplus$, we obtain:
\begin{align*}
\Xplus = \U \Lo \Utr = \sum_{{a,b \geq 0}} \beta^{a+b}{\left(\Hmat\right)^a \M \U {\Lambda^{-(a+b+1)}} \Utr {\M} \left(\Hmat\right)^b}.
\end{align*}
Using triangle inequality, we have:
\begin{align}
  \infnorm{\Xplus - \M}
	  &\leq \infnorm{\M \U \inv{\Lo} \Utr { \M } - \M}
	  + \sum_{\stackrel{a,b \geq 0}{a+b \geq 1}} \abs{\beta}^{a+b} \infnorm{\left(\Hmat \right)^a \M \U {\Lo^{-(a+b+1)}} \Utr \trans{\M} \left(\Hmat \right)^b}.
\label{eqn:update-error-main-app}
\end{align}
Using Lemma~\ref{lem:linf-l2}, we have the following bound for the first term above:
\begin{align}
  \infnorm{\M \U \inv{\Lo} \Utr {\M} - \M}
	&\leq \frac{\muzero^2 r}{n} \twonorm{\M \U \inv{\Lo} \Utr \M - \M }. \label{eqn:linf-ltwo-app}
\end{align}
Furthermore, using Lemma~\ref{lem:cond_spec} we have: 
\begin{align}
  \twonorm{\M \U \inv{\Lo} \Utr \M  - \M} &\leq \abs{\sk[k+1]} + 5 \abs{\beta} \twonorm{\Hmat }, \mbox{ and } \label{eqn:laminvspectral-app} \\
  \twonorm{\M \U \Lo^{-a} \Utr \M} &\leq 4\left(\frac{\abs{\sk}}{2}\right)^{-a+2} \; \forall \; a \geq 2. \label{eqn:laminv-a-spectral-app}
\end{align}
Plugging \eqref{eqn:laminvspectral-app} into \eqref{eqn:linf-ltwo-app} gives us:
\begin{align}\label{eqn:update-error1final-app}
  \infnorm{\M \U \inv{\Lo} \Utr {\M} - \M} &\leq \frac{\muzero^2 r}{n} \left(\abs{\sk[k+1]} + 5 \abs{\beta} \twonorm{\Hmat}\right).
\end{align}

Let $\M = \Uo \So \Uotr$ denote the EVD of $\M$. We now bound the terms in the summation in \eqref{eqn:update-error-main-app} for $1 \leq a+b < \log n$.
\begin{align}
&\abs{\beta}^{a+b}\infnorm{\left(\Hmat\right)^a \M \U {\Lo^{-(a+b+1)}} \Utr {\M} \left(\Hmat\right)^b} \nonumber \\
	&= \abs{\beta}^{a+b} \max_{i,j} \trans{\e_i} \left(\Hmat\right)^a \M \U {\Lo^{-(a+b+1)}} \Utr {\M} \left(\Hmat\right)^b \e_j \nonumber \\
	&\leq \abs{\beta}^{a+b} \left(\max_{i} \trans{\e_i} \left(\Hmat\right)^a \Uo \right) \twonorm{\So \Uotr \U {\Lo^{-(a+b+1)}} \Utr {\Uo \So}} \left(\max_j \Uotr \left(\Hmat\right)^b \e_j\right) \nonumber \\
	&\leq \abs{\beta}^{a+b} \left(\sqrt{r}\max_{i} \infnorm{\left(\Hmat\right)^a \uoi}\right) \twonorm{\M \U {\Lo^{-(a+b+1)}} \Utr {\M}} \left(\sqrt{r}\max_j \infnorm{ \left(\Hmat\right)^b \uoi[j]}\right) \nonumber \\
	&\stackrel{(\zeta_1)}{\leq} \frac{\muzero^2r^2}{n} \abs{\beta}^{a+b} \left(10 \sqrt{\alpha} \log n\right)^{a+b} \twonorm{\M \U {\Lo^{-(a+b+1)}} \Utr {\M}}\nonumber\\
	&\stackrel{(\zeta_2)}{\leq} \frac{\muzero^2r^2}{n} \abs{\beta}^{a+b} \left(10 \sqrt{\alpha} \log n\right)^{a+b} \cdot 4 \left(\frac{2}{ \abs{\sk}}\right)^{a+b-1} \nonumber\\
	&\leq \frac{\muzero^2r^2}{n} \left(\frac{80 \abs{\beta} \sqrt{\alpha} \log n }{\abs{\sk}}\right)^{a+b-1} \left(10 \abs{\beta} \sqrt{\alpha}\log n\right)
	\leq \frac{\muzero^2r^2}{n} \left(\frac{1}{20}\right)^{a+b-1} \cdot 10 \abs{\beta} \sqrt{\alpha}\log n, \label{eqn:Htpq-error1-app}
\end{align}
where $(\zeta_1)$ follows from Lemma~\ref{lem:Htp-spectralnorm_app} and $(\zeta_2)$ follows from~\eqref{eqn:laminv-a-spectral-app}.

For $a+b \geq \log n$, we have
\begin{align}
\abs{\beta}^{a+b}\infnorm{\left(\Hmat\right)^a \M \U {\Lo^{-(a+b+1)}} \Utr { \M } \left(\Hmat\right)^b}
	&\leq \abs{\beta}^{a+b}\twonorm{\left(\Hmat\right)^a \M \U {\Lo^{-(a+b+1)}} \Utr {\M} \left(\Hmat\right)^b} \nonumber \\
	&\leq \abs{\beta}^{a+b} \twonorm{\Hmat}^a \twonorm{ \M \U {\Lo^{-(a+b+1)}} \Utr { \M }} \twonorm{\Hmat}^b \nonumber \\
	&\leq \abs{\beta}^{a+b} \twonorm{\Hmat}^{a+b} \left(\frac{5}{4\abs{\sk}}\right)^{a+b-1} \nonumber \\
	&\leq \left(\frac{15 \abs{\beta} \sqrt{\alpha}}{4\abs{\sk}}\right)^{a+b-1} \cdot 3 \abs{\beta} \sqrt{\alpha} \nonumber \\
	&\leq \frac{\muzero^2 r^2}{n} \left(\frac{1}{20}\right)^{a+b-1} \left(10 \abs{\beta} \sqrt{\alpha} \log n\right),\label{eqn:Htpq-error2-app}
\end{align}
where we used Lemma~\ref{lem:cond_spec} to bound $\twonorm{\M \U {\Lo^{-(a+b+1)}} \Utr { \M }}$
and Lemma~\ref{lem:cond_spec1} to bound $\twonorm{\Hmat}$. The last inequality follows from using $(1/2)^{a+b}\leq 1/n\leq \frac{\mu^2 r^2}{n}$ as $a+b>\log n$. 

Plugging \eqref{eqn:update-error1final-app}, \eqref{eqn:Htpq-error1-app} and \eqref{eqn:Htpq-error2-app}
in \eqref{eqn:update-error-main-app} gives us:
\begin{align*}
  \infnorm{\Xplus - \M} &\leq \frac{\muzero^2 r}{n} \left(\abs{\sk[k+1]} + 5 \abs{\beta} \twonorm{\Hmat} \right)
	+ \frac{\muzero^2 r^2}{n} \sum_{\stackrel{a,b \geq 0}{a+b \geq 1}} \left(\frac{1}{20}\right)^{a+b} (10 \abs{\beta} \sqrt{\alpha} \log n) \\
	&\leq \frac{\muzero^2 r^2}{n} \left(\abs{\sk[k+1]} + 15 \abs{\beta} \sqrt{\alpha} \log n \right).
\end{align*}
%
This proves the lemma.

\subsection{Proofs of Technical Lemmas from Section~\ref{sec:Hmat}, Section~\ref{sec:tech-error-decay-general}}\label{sec:techproof-error-decay-general}
\begin{proof}[Proof of Lemma~\ref{lem:sat-defn}]
Since $(\PoA)_{ij}$ is an unbiased estimate of $\A_{ij}$, we see that $\expec{\mB_{ij}} = 0$.
For $k\geq 2$, we have:
\begin{align*}
\expec{\abs{\mB_{ij}}^k} = \left(\frac{\sqrt{p} \A_{ij}}{2 \sqrt{n} \infnorm{\A}}\right)^k \left( p \left(\frac{1}{p}-1\right)^k + (1-p) \right)
	\leq \left(\frac{p}{2n}\right)^{\frac{k}{2}} \; \cdot  \; \frac{2}{p^{k-1}}
	\leq \frac{1}{n\left(np\right)^{\frac{k}{2}-1}} \leq \frac{1}{n}.
\end{align*}
\end{proof}
\begin{proof}[Proof of Lemma~\ref{lem:cond_spec1}]

  Note that, $\Hmat=\sum_{ij} h_{ij} \e_i \trans{\e_j}=\sum_{i\leq j} \G_{ij}$ where $\G_{ij}=h_{ij} \frac{\ind{i\neq j}+1}{2}
\left(\e_i \trans{\e_j}+\e_j \trans{\e_i} \right)$. Now, $\expec{\G_{ij}}=0$, $\max_{ij}\twonorm{\G_{ij}}=2$, and,
$$\twonorm{\expec{\G_{ij}\trans{\G}_{ij}}} = \twonorm{\expec{\sum_{ij}h_{ij}^2 \e_i \trans{\e_i}}}
= \max_i \sum_j \expec{h_{ij}^2} \leq 1.$$
The lemma now follows using matrix Bernstein inequality (Lemma~\ref{lem:matbern}).
\end{proof}

\begin{proof}[Proof of Lemma~\ref{lem:linf-l2}]
  Let $\M = \Us \So \Ustr$ be the eigenvalue decomposition  $\M$. We have:
\begin{align*}
  \infnorm{\M {\mB} \M - \M} &= \max_{i,j} \trans{\ei} \left(\M {\mB} \M - \M\right) \ej \\
	&= \max_{i,j} \trans{\ei} \left(\Us \So \Utr {\mB} \Us \So \Ustr - \Us \So \Ustr \right) \ej \\
	&\leq \left(\max_{i} \twonorm{\trans{\ei} \Us}\right) \twonorm{\So \Ustr {\mB} \Us \So - \So} \left(\max_j \Ustr \ej\right) \\
	&\stackrel{(\zeta_1)}{\leq} \frac{\mu^2 r}{n} \twonorm{\Us \left(\So \Ustr {\mB} \Us \So - \So \right) \Ustr}
	= \frac{\mu^2 r}{n} \twonorm{\M {\mB} \M - \M},
\end{align*}
where $(\zeta_1)$ follows from the incoherence of $\M$.
\end{proof}

\begin{proof}[Proof of Lemma~\ref{lem:cond_spec}]
Let $\W=\U \Lo \trans{\U} + \widetilde{\U} \widetilde{\Lo} \trans{\widetilde{\U}}$ be the eigenvalue decomposition of
$\W$. Since $\Pk(\W) = \U \Lo \trans{\U}$, we see that $\abs{\lambda_k} \geq \abs{\ltili}$.

From Lemma~\ref{lem:weyl-perturbation}, we have:
\begin{equation}
  \label{eq:cond_pert1}
  |\lambda_i-\beta_i|\leq \twonorm{\mE},\ \forall \ i \in [k],\ \ \ \text{and},\ \ \  \abs{\ltili - \beta_{k+i}}\leq \twonorm{\mE},
\ \forall \ i\in [n-k].
\end{equation}
Since $\twonorm{\mE} \leq \frac{\beta_k}{2}$, we see that
\begin{align}\label{eqn:lamkbound}
\abs{\lambda_k} \geq \abs{\beta_k}/2>0.
\end{align}
Hence, we conclude that $\Lo \in \R^{k\times k}$ is invertible proving the first claim of the lemma.

Using the eigenvalue decomposition of $\W$, we have the following expansion: 
\begin{align}
  \A \U \inv{\Lo} \Utr \A  - \A
	&= \left(\U \Lo \Utr + \Util \widetilde{\Lo} \Utiltr - \mE \right) \U \inv{\Lo} \Utr \left(\U \Lo \Utr + \Util \widetilde{\Lo} \Utiltr - \mE \right) - \A \nonumber\\
	&= \U \Lo \Utr - \U \Utr \mE - \mE \U \Utr + \mE \U \inv{\Lo} \Utr \mE  - \U \Lo \Utr - \Util \widetilde{\Lo} \Utiltr + \mE \nonumber\\
	&= - \U \Utr \mE - \mE \U \Utr + \mE \U \inv{\Lo} \Utr \mE- \Util \widetilde{\Lo} \Utiltr + \mE. \label{eq:cond_spec_3}
\end{align}
Applying triangle inequality and using $\twonorm{\mB\C} \leq \twonorm{\mB} \twonorm{\C}$, we get: 
\begin{align*}
  \twonorm{\A-\A \U\Lo^{-1} \Utr \trans{\A} } \leq 3\twonorm{\mE} + \frac{\twonorm{\mE}^2}{|\lambda_k|} + \abs{\ltili[1]}.
\end{align*}
Using the above inequality with \eqref{eqn:lamkbound}, we obtain: 
\begin{align*}
  \twonorm{\A-\A \U\Lo^{-1} \Utr \trans{\A}} \leq \abs{\beta_{k+1}} + 5\twonorm{\mE}.
\end{align*}
This proves the second claim of the lemma.

Now, similar to \eqref{eq:cond_spec_3}, we have: 
\begin{align*}
  \A \U \Lo^{-a} \Utr \A
	&= \left(\U \Lo \Utr + \Util \widetilde{\Lo} \Utiltr - \mE\right) \U  \Lo^{-a} \Utr \left(\U \Lo \Utr + \Util \widetilde{\Lo} \Utiltr - \mE \right) \nonumber\\
	&= \U \Lo^{-a+2} \Utr - \U \Lo^{-a+1}\Utr \mE - \mE \U \Lo^{-a+1}\Utr + \mE \U \Lo^{-a} \Utr \mE.
\end{align*}
The last claim of the lemma follows by using triangle inequality and~\eqref{eqn:lamkbound} in the above equation. 
\end{proof}

\section{Proof of Lemma~\ref{lem:error-decay-samesamples}}
We now present a proof of Lemma~\ref{lem:error-decay-samesamples} that show decrease in the Frobenius norm of the error matrix, despite using same samples in each iteration. In order to state our proof, we will first introduce certain notations and provide a few perturbation results that might be of independent interest. Then, in next subsection, we will present a detailed proof of Lemma~\ref{lem:error-decay-samesamples}. Finally, in Section~\ref{sec:techproof-error-decay-samesamples}, we present proofs of the technical lemmas given below. 

\subsection{Notations and Technical Lemmas}\label{sec:tech-error-decay-samesamples}
Recall that we assume (wlog) that $\M\in \R^{n\times n}$ is symmetric and $\M=\Us\So\Ustr$ is the eigenvalue decomposition (EVD) of $\M$.

In order to state our first supporting lemma, we will introduce the concept of tangent spaces of matrices \cite{bhatia}.
\begin{definition}\label{defn:tangentspace}
Let $\A$ be a matrix with EVD (eigenvalue decomposition) $\Us \So \Ustr$. The following space of matrices is called the tangent space of $\A$:
\begin{align*}
\cT(\A) \defas \set{\Us \Lo_0 \Ustr + \Us \Lo_1 \Utr + \U \Lo_2 \Ustr},
\end{align*}
where $\U\in \R^{n\times n}, \U^T\U=I$, and $\Lo_0, \Lo_1, \Lo_2$ are all diagonal matrices.
\end{definition}
That is, if $\A = \Us \So \Ustr$ is the EVD of $\A$, then any matrix $\mB$ can be decomposed into four mutually orthogonal terms as
\begin{align}
\mB = \Us \Ustr \mB \Us \Ustr + \Us \Ustr \mB \Usperp \Usperptr + \Usperp \Usperptr \mB \Us \Ustr
	+ \Usperp \Usperptr \mB \Usperp \Usperptr,
\label{eqn:matrix-orth-decomp}
\end{align}
where $\Usperp$ is a basis of the orthogonal space of $\Us$. The first three terms above are in $\cT(\A)$ and the last term is in ${\cT(\A)}^{\perp}$. We let $\cP_{\cT(\A)}$ and $\cP_{\cT(\A)^{\perp}}$ denote the projection operators onto $\cT(\A)$ and $\cT(\A)^{\perp}$ respectively.
\begin{lemma}\label{lem:perturbation_perpspace}
Let $\A$ and $\mB$ be two symmetric matrices. Suppose further that $\mB$ is rank-$k$. Then, we have:
\begin{align*}
\frob{\cP_{\cT(\A)^{\perp}} \left( \mB \right)} \leq \frac{\frob{\A-\mB}^2}{{\sigma_k(\mB)}}.
\end{align*}
\end{lemma}
Next, we present a few technical lemmas related to norm of $M-P_{\Omega}(M)$:
\begin{lemma}\label{lem:bound_tgt1}
Let $M$, $\Omega$ be as given in Lemma~\ref{lem:error-decay-samesamples} and let $p=|\Omega|/n^2$ be the sampling probability. Then, 
For every $r \times r$ matrix $\Sh$, we have $(w.p. \geq 1-n^{-10-\alpha})$:
\begin{align*}
\frob{\left(\Us \Sh \Ustr - \frac{1}{p}\Om{\Us \Sh \Ustr}\right) \Us} \leq \frac{1}{40} \frob{\Sh}.
\end{align*}
\end{lemma}

\begin{lemma}\label{lem:bound_tgt2}
Let $M$, $\Omega$, $p$ be as given in Lemma~\ref{lem:error-decay-samesamples}. Then, for every $i,j \in [r]$, we have $(w.p. \geq 1-n^{-10-\alpha})$:
\begin{align*}
\twonorm{\uoi[j] \uoitr - \frac{1}{p} P_{\Omega}\left( \uoi[j] \uoitr \right)} < \frac{1}{40r\sqrt{r}}.
\end{align*}
\end{lemma}

\begin{lemma}\label{lem:bound_tgt3}
Let $M$, $\Omega$, $p$ be as given in Lemma~\ref{lem:error-decay-samesamples}. Then, for every $i,j \in [r]$ and $s \in [n]$, we have $(w.p. \geq 1-n^{-10-\alpha})$:
\begin{align*}
\abs{\ip{\uoi}{\uoi[j]} - \frac{1}{p} \sum_{(s,l) \in \Omega} \left(\uoi\right)_l \left(\uoi[j]\right)_l} < \frac{1}{40r\sqrt{r}}.
\end{align*}
\end{lemma}

\subsection{Detailed Proof of Lemma~\ref{lem:error-decay-samesamples}}\label{sec:mainproof-error-decay-samesamples}
Let $\Em \defas {\X}-{\PkM}$, $\Hmat \defas \Em - \frac{1}{p} \Om{\Em}$ and $\G \defas \X - \frac{1}{p} \Om{\X - \M} = \PkM + \Hmat - \frac{1}{p}\Om{\M - \PkM}$.
That is, $\Xplus = \Pk\left(\G \right)$.

For simplicity, {\em in this section}, we let $\M = \Us \So \Ustr + \Usperp \Sob \Usperptr$ denote the eigenvalue decomposition (EVD) of $\M$ with $\PkM = \Us \So \Ustr$, and also let $\Mb = \Usperp \Sob \Usperptr$.
We also use the shorthand notation $\cT \defas \cT\left(\PkM\right)$.

Representing $\X$ in terms of its projection onto $\cT$ and its complement, we have:
\begin{align}
\X = \Us \Lo_0 \Ustr + \Us \Lo_1 \Usperptr + \Usperp \trans{\Lo_1} \Ustr + \Usperp \Lo_3 \Usperptr,
\label{eqn:Xt-decomp}
\end{align}
and also conclude that: 
\begin{align*}
\frob{\So - \Lo_0} \leq \frob{\X - \PkM}, \quad \frob{\Lo_1} \leq \frob{\X - \PkM}, \mbox{ and }\quad
\frob{\Lo_3} \leq 
\frac{\frob{\X - \PkM}}{n^2},
\end{align*}
where the last conclusion follows from Lemma~\ref{lem:perturbation_perpspace} and the hypothesis that $\frob{\X - \PkM} < \frac{\abs{\sk}}{n^2}$.
Using $\|\Em\|_F\leq \sigma_k/n^2$, we have: 
\begin{align*}
\frob{\Hmat} \leq \frac{2}{p} \frob{\Em} \leq \frac{2}{p} \frac{\sk}{n^2} &\leq \frac{\sk}{8},
\mbox{ and,}\\
\frob{\frac{1}{p}\Om{\Mb}} \leq \frac{1}{p}\frob{\Mb} &\leq \frac{1}{p} \frac{\sk}{n^2} \leq \frac{\sk}{8},
\end{align*}
where we used the hypothesis that $\frob{\M - \PkM} < \frac{\sk}{n^2}$ in the second inequality.

The above bounds implies: 
\begin{align}\label{eq:lem2_tmp_1}
\frob{\cPT{\Hmat - \frac{1}{p}\Om{\Mb} }} \leq \frob{{\Hmat - \frac{1}{p}\Om{\Mb} }}
\leq \twonorm{\Hmat} + \frob{\frac{1}{p}\Om{\Mb}} \leq \frac{\sk}{4}.
\end{align}
Similarly,  \begin{equation}\label{eq:lem2_tmp_2}\twonorm{\cPTp{\Hmat - \frac{1}{p}\Om{\Mb} }}\leq \frac{\sk}{4}.\end{equation}

Since $\Xplus = P_k\left(\PkM + \Hmat - \frac{1}{p} \Om{\Mb}\right)$, using Lemma~\ref{lem:daviskahan-approx} with \eqref{eq:lem2_tmp_1}, \eqref{eq:lem2_tmp_2}, we have: 
\begin{align*}
\frob{\PkM - \Xplus} &= \frob{P_k \left(\PkM+\cPTp{\Hmat - \frac{1}{p}\Om{\Mb} } \right) - P_k\left(\PkM + \Hmat - \frac{1}{p}\Om{\Mb} \right)} \\
	&\leq c \frob{\cPT{\Hmat - \frac{1}{p}\Om{\Mb} }}.
\end{align*}
Now, using Claim~\ref{cl:lem2_tmp2}, we have $\frob{\cPT{\Hmat - \frac{1}{p}\Om{\Mb} }} < \frac{1}{10} \frob{\PkM - \X} + \frac{2}{\sqrt{p}} \frob{\Mb}$, which along with the above equation {\em establishes the lemma}. We now state and prove the claim bounding $\frob{\cPT{\Hmat - \frac{1}{p}\Om{\Mb} }}$ that we used above to finish the proof. \\

\begin{claim}\label{cl:lem2_tmp2}
Assume notation defined in the section above. Then, we have: 
$$\frob{\cPT{\Hmat - \frac{1}{p}\Om{\Mb} }} < \frac{1}{10} \frob{\PkM - \X} + \frac{2}{\sqrt{p}} \frob{\Mb}.$$
\end{claim}
\begin{proof}
We first bound $\frob{\cPT{\Hmat}}$.
Recalling that $\PkM = \Us \So \Ustr$ is the EVD of $\PkM$, we have:
\begin{align*}
\frob{\cPT{\Hmat}} < 2 \frob{\Hmat \Us}.
\end{align*}
Using \eqref{eqn:Xt-decomp}, we have:
\begin{align}
\Hmat \Us &= \left[\Us \left(\So - \Lo_0\right) \Ustr -
		\frac{1}{p} P_{\Omega}\left(\Us \left(\So - \Lo_0\right) \Ustr\right) \right] \Us
		+ \left[ \Us \Lo_1 \Usperptr - \frac{1}{p} P_{\Omega}\left( \Us \Lo_1 \Usperptr \right) \right] \Us \nonumber \\
		&+ \left[ \Usperp \Lo_2 \Ustr - \frac{1}{p} P_{\Omega}\left( \Usperp \Lo_2 \Ustr \right) \right] \Us
		+ \left[  \Usperp \Lo_3 \Usperptr - \frac{1}{p} P_{\Omega}\left( \Usperp \Lo_3 \Usperptr \right) \right] \Us.
\label{eqn:Tgt_error}
\end{align}

\textbf{Step I}: To bound the first term in \eqref{eqn:Tgt_error}, we use Lemma~\ref{lem:bound_tgt1} to obtain:
\begin{align}
\frob{\left(\Us \left(\So-\Lo_0\right) \Ustr - \frac{1}{p}\Om{\Us \left(\So-\Lo_0\right) \Ustr}\right) \Us} \leq \frac{1}{40} \frob{\left(\So-\Lo_0\right)} \leq \frac{1}{40} \frob{\X-\PkM}.
\label{eqn:errordecay-term1}
\end{align}

\textbf{Step II}: To bound the second term, we let $\U \defas \Usperp \trans{\Lo_1}$, and proceed as follows:
\begin{align*}
&\hspace*{-50pt}\twonorm{\left[ \Us \Lo_1 \Usperptr - \frac{1}{p} P_{\Omega}\left( \Us \Lo_1 \Usperptr \right) \right] \uoi}
= \twonorm{\sum_{j=1}^r \left( \uoi[j] \uitr[j] - \frac{1}{p} P_{\Omega}\left( \uoi[j] \uitr[j] \right) \right) \uoi} \\
&= \twonorm{\sum_{j=1}^r \left( \uoi[j] \uoitr - \frac{1}{p} P_{\Omega}\left( \uoi[j] \uoitr \right) \right) \ui[j]} \leq \sum_{j=1}^r \twonorm{\uoi[j] \uoitr - \frac{1}{p} P_{\Omega}\left( \uoi[j] \uoitr \right) } \twonorm{\ui[j]} \\
&\stackrel{(\zeta_1)}{\leq} \frac{1}{40r\sqrt{r}}\sum_{j=1}^r \twonorm{\ui[j]}
\leq  \frac{1}{40\sqrt{r}} \frob{\Lo_1 \Usperptr} \leq \frac{1}{40\sqrt{r}} \frob{\PkM - \X},
\end{align*}
where $(\zeta_1)$ follows from Lemma~\ref{lem:bound_tgt2}.
This means that we can bound the second term as:
\begin{align}
\frob{\left[ \Us \Lo_1 \Usperptr - \frac{1}{p} P_{\Omega}\left( \Us \Lo_1 \Usperptr \right) \right] \Us}
\leq \frac{1}{40} \frob{\PkM - \X}.
\label{eqn:errordecay-term2}
\end{align}

\textbf{Step III}: We now let $\U \defas \Usperp \Lo_2$ and turn to bound the third term in \eqref{eqn:Tgt_error}. We have:
\begin{align*}
&\hspace*{-50pt}\twonorm{\left[ \Usperp \Lo_2 \Ustr - \frac{1}{p} P_{\Omega}\left( \Usperp \Lo_2 \Ustr \right) \right] \uoi}
= \twonorm{\sum_{j=1}^r \ip{\uoi[j]}{\uoi} \ui[j] - \ui[j] \odot \e_{\ip{\uoi[j]}{\uoi}_{\Omega}}} \\
&= \twonorm{\sum_{j=1}^r \ui[j] \odot \left(\ip{\uoi[j]}{\uoi} \ones - \e_{\ip{\uoi[j]}{\uoi}_{\Omega}}\right)} \stackrel{(\zeta_1)}{\leq} \frac{1}{40r\sqrt{r}} \sum_{j=1}^r \twonorm{\ui[j]}\\
&\leq \frac{1}{40\sqrt{r}} \frob{\Usperp \Lo_2} \leq \frac{1}{40\sqrt{r}} \frob{\PkM - \X},
\end{align*}
where $\ones$ denotes the all ones vector, and $\e_{\ip{\uoi[j]}{\uoi}_{\Omega}}$ denotes a vector whose $s^{\textrm{th}}$ coordinate is given by $\frac{1}{p} \sum_{l: (s,l)\in \Omega} \left(\uoi[j]\right)_l \left(\uoi\right)_l$. Note that $(\zeta_1)$ follows from Lemma~\ref{lem:bound_tgt3}.
So, we again have:
\begin{align}
\frob{\left[ \Usperp \Lo_2 \Ustr - \frac{1}{p} P_{\Omega}\left( \Usperp \Lo_2 \Ustr \right) \right] \Us}
\leq \frac{1}{40} \frob{\PkM - \X}.
\label{eqn:errordecay-term3}
\end{align}

\textbf{Step IV}: To bound the last term in \eqref{eqn:Tgt_error}, we use Lemma~\ref{lem:perturbation_perpspace} to conclude
\begin{align}
&\hspace*{-27pt}\frob{\left[  \Usperp \Lo_3 \Usperptr - \frac{1}{p} P_{\Omega}\left( \Usperp \Lo_3 \Usperptr \right) \right] \Us}
\leq \frob{\Usperp \Lo_3 \Usperptr - \frac{1}{p} P_{\Omega}\left( \Usperp \Lo_3 \Usperptr \right) } \nonumber\\
&\leq \frac{2}{p} \frob{ \Usperp \Lo_3 \Usperptr }
= \frac{2}{p} \frob{ \cPTp{ \X } } \leq \frac{2}{p} \frac{\frob{\PkM - \X}^2}{{\sigma_k(\X)}} \leq \frac{1}{40n} \frob{\PkM-\X}.
\label{eqn:errordecay-term4}
\end{align}

Combining \eqref{eqn:errordecay-term1}, \eqref{eqn:errordecay-term2}, \eqref{eqn:errordecay-term3} and \eqref{eqn:errordecay-term4}, we have:
\begin{align}
\frob{\cPT{\Hmat}} \leq \frac{1}{10} \frob{\PkM - \X}.
\label{eqn:frob-Hmat}
\end{align}

On the other hand, we trivially have:
\begin{align}
\frob{\frac{1}{p}\Om{\Mb}} \leq \frac{2}{{p}} \frob{\Mb}.
\label{eqn:frob-Mb}
\end{align}

Claim now follows by combining~\eqref{eqn:frob-Hmat} and~\eqref{eqn:frob-Mb}.
%
\end{proof}
\subsection{Proofs of Technical Lemmas from Section~\ref{sec:tech-error-decay-samesamples}}\label{sec:techproof-error-decay-samesamples}
\begin{proof}[Proof of Lemma~\ref{lem:perturbation_perpspace}]
Let $B=\U\Lo\Utr$ be EVD of $B$. Then, we have: 
\begin{align*}
\frob{\cP_{\cT(\A)^{\perp}} \left( \mB \right)}&=\frob{\Usperp \Usperptr \mB \Usperp \Usperptr} = \frob{\Usperptr \U \Lo \Utr \Usperp} = \frob{\Usperptr \U \Lo \inv{\Lo} \Lo \Utr \Usperp} \\
	&\leq \frob{\Usperptr \U \Lo} \twonorm{\inv{\Lo}} \frob{\Lo \Utr \Usperp}= \frob{\A-\mB} \twonorm{\inv{\Lo}} \frob{\A-\mB}\\
	&\leq \frac{\frob{\A - \mB}^2}{{\sigma_k(\mB)}}.
\end{align*}
Hence Proved.
\end{proof}

We will now prove Lemma~\ref{lem:daviskahan-approx}, which is a natural extension of the Davis-Kahan theorem. In order to do so, we will first recall the Davis-Kahan theorem:
\begin{theorem}[Theorem VII.3.1 of \cite{bhatia}]\label{thm:davis-kahan-bhatia}
Let $\A$ and $\mB$ be symmetric matrices. Let $S_1, S_2 \subseteq \R$ be subsets separated by $\nu$. Let $\Em = P_{\A}(S_1)$ and $\Fm = P_{\mB}(S_2)$ be an orthonormal basis of the eigenvectors of $\A$ with eigenvalues in $S_1$ and that of the eigenvectors of $\mB$ with eigenvalues in $S_2$ respectively. Then, we have:
\begin{align*}
\twonorm{\Em\Fm} &\leq \frac{1}{\nu} \twonorm{\A-\mB},\ \ \ \frob{\Em \Fm} \leq \frac{1}{\nu} \frob{\A-\mB}.
\end{align*}
\end{theorem}
\begin{proof}[Proof of Lemma~\ref{lem:daviskahan-approx}]
Let $\A = \Us \So \Ustr + \Usperp \Sh \Usperptr$ be the EVD of $\A$ with $P_k\left(\A\right) = \Us \So \Ustr $. Similarly, let $\A + \Em = \U \Lo \Utr + \Uperp \Lhat \Uperptr$ denote the EVD of $\A + \Em$ with $P_k\left(\A+\Em\right) = \U \Lo \Utr$. Expanding $P_k\left(\A+\Em\right)$ into components along $\Us$ and orthogonal to it, we have:
\begin{align*}
\U \Lo \Utr = \Us \Ustr \U \Lo \Utr \Us \Ustr + \Usperp \Usperptr \U \Lo \Utr \Us \Ustr + \U \Lo \Utr \Usperp \Usperptr.
\end{align*}
Now, 
\begin{align}
&\frob{P_k\left(\A+\Em\right) - P_k\left(\A\right)} \nonumber\\
&= \frob{\Us \Ustr \U \Lo \Utr \Us \Ustr + \Usperp \Usperptr \U \Lo \Utr \Us \Ustr + \U \Lo \Utr \Usperp \Usperptr - \Us \So \Ustr} \nonumber\\
&\leq \frob{\Us \Ustr \U \Lo \Utr \Us \Ustr - \Us \So \Ustr} + \frob{\Usperp \Usperptr \U \Lo \Utr \Us \Ustr} + \frob{ \U \Lo \Utr \Usperp \Usperptr } \nonumber\\
&\leq \frob{\Us \Ustr \U \Lo \Utr \Us \Ustr - \Us \So \Ustr} + \frob{\Usperp \Usperptr \U \Lo \Utr} + \frob{ \U \Lo \Utr \Usperp \Usperptr } \nonumber\\
&= \frob{\Us \Ustr \U \Lo \Utr \Us \Ustr - \Us \So \Ustr} + 2 \frob{ \U \Lo \Utr \Usperp \Usperptr } \nonumber \\
&\leq \frob{\Us \Ustr \U \Lo \Utr \Us \Ustr + \Us \Ustr \Uperp \Lhat \Uperptr \Us \Ustr - \Us \So \Ustr} \nonumber \\
&\quad + \frob{\Us \Ustr \Uperp \Lhat \Uperptr \Us \Ustr}
+ 2 \frob{ \Lo \Utr \Usperp } \nonumber \\
&= \frob{\Us \Ustr \Em \Us \Ustr } + \frob{\Ustr \Uperp \Lhat \Uperptr \Us }
+ 2 \frob{ \Lo \Utr \Usperp } \nonumber \\
&\leq \frob{\Em} + \frob{\Ustr \Uperp \Lhat \Uperptr \Us }
+ 2 \frob{ \Lo \Utr \Usperp }
\label{eqn:davis-kahan-bd1}
\end{align}

Before going on to bound the terms in \eqref{eqn:davis-kahan-bd1},
let us make some observations. We first use Lemma~\ref{lem:weyl-perturbation} to conclude that
\begin{align*}
\frac{3}{4}\abs{\sk[i]} \leq \abs{\lami} \leq \frac{5}{4} \abs{\sk[i]}, \mbox{ and }
\abs{\lhati[k+i]} \leq \frac{\abs{\sk}}{2}.
\end{align*}
Applying Theorem~\ref{thm:davis-kahan-bhatia} with $S_1 = \left[ \frac{-\abs{\sk}}{2}, \frac{\abs{\sk}}{2} \right]$ and $S_2 = \left(-\infty, \frac{-3\abs{\sk[i]}}{4}\right] \cup \left[ \frac{3\abs{\sk[i]}}{4},\infty\right)$, with separation parameter $\nu = \frac{\abs{\sk[i]}}{4}$, we see that
\begin{align}
\twonorm{\uitr \Usperp} &\leq \frac{4}{\abs{\sk[i]}}\twonorm{\Em},
\mbox{ and }\label{eqn:uiusperp} \\
\frob{\Uperptr \Us} &\leq \frac{4}{\abs{\sk}} \frob{\Em}. \label{eqn:uperpus}
\end{align}

We are now ready to bound the last two terms in the right hand side of \eqref{eqn:davis-kahan-bd1}. Firstly, we have:
\begin{align*}
\frob{\Ustr \Uperp \Lhat \Uperptr \Us }
\leq \twonorm{\Lhat } \twonorm{\Ustr \Uperp} \frob{\Uperptr\Us}
\leq \abs{\lhati[k+1]} \frob{\Uperptr\Us}^2 \leq 2 \frob{\Em},
\end{align*}
where the last step follows from~\eqref{eqn:uperpus} and the assumption on $\frob{\Em}$. For the other term, we have:
\begin{align*}
\frob{ \Lo \Utr \Usperp }^2 = \sum_i \lami^2 \twonorm{\uitr \Usperp}^2 \leq \frac{25}{16} \sum_i \sk[i]^2 \frac{16 \twonorm{\Em}^2}{\sk[i]^2} = 25 k \twonorm{\Em}^2,
\end{align*}
where we used~\eqref{eqn:uiusperp}. Combining the above two inequalities with~\eqref{eqn:davis-kahan-bd1} proves the lemma.
\end{proof}
Finally, we present proofs for Lemma~\ref{lem:bound_tgt1}, Lemma~\ref{lem:bound_tgt2}, Lemma~\ref{lem:bound_tgt3}. 
\begin{proof}[Proof of Lemma~\ref{lem:bound_tgt1}]
Using Theorem 1 by \cite{BhojanapalliJ14}, the followings $\forall \widehat{\Sigma}$ (w.p. $\geq 1-n^{-10-\alpha}$): 
  $$\twonorm{\Us \Sh \Ustr - \frac{1}{p}\Om{\Us \Sh \Ustr}}\leq \frac{\mu^2r}{\sqrt{np}}\|\Sh\|_2\leq \frac{1}{\sqrt{r\cdot C\cdot \alpha}\log n}\|\Sh\|_2.$$
Lemma now follows by using the assumed value of $p$ in the above bound along with the fact that $\left(\Us \Sh \Ustr - \frac{1}{p}\Om{\Us \Sh \Ustr}\right)\Us$ is a rank-$r$ matrix. 
\end{proof}
\begin{proof}[Proof of Lemma~\ref{lem:bound_tgt2}]
  Let $H=\frac{1}{\beta}\left(\uoi[j] \uoitr - \frac{1}{p} P_{\Omega}\left( \uoi[j] \uoitr \right)\right)$, where $\beta=\frac{2\mu^2r}{\sqrt{n\cdot p}}$. Then, using Lemma~\ref{lem:sat-defn}, $H$ satisfies the conditions of Definition~\ref{defn:moment-condns}. Lemma now follows by using Lemma~\ref{lem:cond_spec1} and using $p$ as given in the lemma statement. 
\end{proof}
\begin{proof}[Proof of Lemma~\ref{lem:bound_tgt3}]
  Let $\delta_{ij}=\mathbb{I}[(i,j)\in \Omega]$. Then, \begin{equation}\label{eq:tgt3_1}\ip{\uoi}{\uoi[j]}-\frac{1}{p} \sum_{(s,l) \in \Omega} \left(\uoi\right)_l \left(\uoi[j]\right)_l=\sum_{l}(1-\frac{\delta_{sl}}{p})\left(\uoi\right)_l \left(\uoi[j]\right)_l=\sum_{l} B_{l},\end{equation}
where $\E[B_{l}]=0$, $|B_{l}|\leq \frac{2\mu^2r}{n\cdot p}$, and $\sum_\E[B_{l}^2]=\frac{\mu^2r}{n\cdot p}$. Lemma follows by using  Bernstein inequality (given below) along with the sampling probability $p$ specified in the lemma. 
\end{proof}

\begin{lemma}[Bernstein Inequality]
 Let $b_{i}$ be a set of independent bounded random variables, then the following holds $\forall\ t>0$: 
$$Pr\left(\left|\sum_{i=1}^n b_i-\E[\sum_{i=1}^n b_i]\right|\geq t\right)\leq \exp\left(-\frac{t^2}{\E[\sum_i b_i^2]+t\max_i|b_i|/3}\right).$$
\end{lemma}
\begin{lemma}[Matrix Bernstein Inequality (Theorem 1.4 of \cite{tropp})]\label{lem:matbern}
 Let $B_{i}\in \R^{n\times n}$ be a set of independent bounded random matrices, then the following holds $\forall\ t>0$: 
 $$Pr\left(\twonorm{\sum_{i=1}^n B_i-\E[\sum_{i=1}^n B_i]}\geq t\right)\leq n\exp\left(-\frac{t^2}{\sigma^2+tR/3}\right),$$
where $\sigma^2=\E\left[\sum_i B_i^2\right]$ and $R=\max_i \|B_i\|_2$. 
\end{lemma}

\section{Proof of Lemma~\ref{lem:Htp-spectralnorm_app}}
\label{app:Hpbound}

We will prove the statement for $r = 1$. The lemma can be proved by taking a union bound over all $r$.
In order to prove the lemma, we will calculate a high order moment of the random variable
\begin{align*}
\wX_a \defas {\ip{\e_1}{\wH^a \u}},
\end{align*}
and then use Markov inequality.
We use the following notation which is mostly consistent with Lemma~$6.5$ of \cite{ErdosKYY2012}. We abbreviate
$(i,j)$ as $\alpha$ and denote $\wh{h}_{ij}$ by $\wh{h}_{\alpha}$. We further let
\begin{align*}
  B_{(i,j)(k,l)} \defas \delta_{jk}.
\end{align*}
With this notation, we have:
\begin{align*}
\wX_a = \sum_{\stackrel{\alpha_1,\cdots,\alpha_a}{\alpha_1(1)=1}} B_{\alpha_1 \alpha_2}
	\cdots B_{\alpha_{a-1} \alpha_a} {\wh{h}_{\alpha_1}\cdots \wh{h}_{\alpha_a}} u_{\alpha_a(2)}.
\end{align*}
We now split the matrix $\wH$ into two parts $\Hmat$ and $\Hmat'$ which correspond to the
upper triangular and lower triangular parts of $\wH$.
This means
\begin{align}
\wX_a = \sum_{\stackrel{\alpha_1,\cdots,\alpha_a}{\alpha_1(1)=1}} B_{\alpha_1 \alpha_2}
	\cdots B_{\alpha_{a-1} \alpha_a} {\left(h_{\alpha_1} + h'_{\alpha_1}\right)\cdots \left(h_{\alpha_a} + h'_{\alpha_a}\right)} u_{\alpha_a(2)}.
\label{eqn:split-up-low}
\end{align}
The above summation has $2^a$ terms, of which we consider only
\begin{align*}
{X}_a \defas \sum_{\stackrel{\alpha_1,\cdots,\alpha_a}{\alpha_1(1)=1}} B_{\alpha_1 \alpha_2}
	\cdots B_{\alpha_{a-1} \alpha_a} {h_{\alpha_1} \cdots h_{\alpha_a}} u_{\alpha_a(2)}.
\end{align*}
The resulting factor of $2^a$ does not change the result.

Abbreviating $\balpha \defas (\alpha_1,\cdots,\alpha_a)$, and
\begin{align*}
  \zeta_{\balpha} \defas B_{\alpha_1 \alpha_2}
	\cdots B_{\alpha_{a-1} \alpha_a} {h_{\alpha_1}\cdots h_{\alpha_a}} u_{\alpha_a(2)},
\end{align*}
we can write
\begin{align*}
  X_a = \sum_{\balpha} \zeta_{\balpha},
\end{align*}
where the summation runs only over those $\balpha$ such that $\alpha_1(1) = 1$.

Calculating the $k^{\textrm{th}}$ moment expansion of $X_a$ for some even number $k$, we obtain:
\begin{align}
\expec{X_a^k} = \sum_{\balpha^1,\cdots,\balpha^k} \expec{\zeta_{\balpha^1}\cdots
	\zeta_{\balpha^k}}. \label{eqn:Xpk_sum}
\end{align}

%
For each valid $\balpha = (\balpha^s) = (\alpha_l^s)$, we define the partition $\Gamma(\balpha)$ of the index set
$\set{(s,l): s \in [k]; l \in [a]}$, where $(s,l)$ and $(s',l')$ are in the same equivalence class if
$\alpha_l^s = \alpha_{l'}^{s'}$. We first bound the contribution of all $\balpha$ corresponding to a partition $\Gamma$ in
the summation \eqref{eqn:Xpk_sum} and then bound the total number of partitions $\Gamma$ possible.
Since each $h_{\alpha}$ is centered, we can conclude that any partition $\Gamma$ that has a non-zero contribution to
the summation in \eqref{eqn:Xpk_sum} satisfies:
\begin{itemize}
  \item[(*)]	each equivalence class of $\Gamma$ contains at least two elements.
\end{itemize}
We further bound the summation in \eqref{eqn:Xpk_sum} by taking absolute values of the summands
\begin{align}
  {\expec{X_a^k}} &\leq \sum_{\balpha^1,\cdots,\balpha^k} \expec{\abs{\zeta_{\balpha^1}} \cdots
	\abs{\zeta_{\balpha^k}}},
	\label{eqn:Xpk_sumbound}
\end{align}
where the summation runs over $(\balpha^1,\cdots,\balpha^k)$ that correspond to valid partitions $\Gamma$.
Fixing one such partition $\Gamma$, we bound the contribution to \eqref{eqn:Xpk_sumbound} of all the terms $\balpha$ such
that $\Gamma(\balpha) = \Gamma$.

We denote $G \equiv G(\Gamma)$ to be the graph constructed from $\Gamma$ as follows. The vertex set $V(G)$ is given
by the equivalence classes of $\Gamma$. For every $(s,l)$, we have an edge between the equivalence class of $(s,l)$ and the
equivalence class of $(s,l+1)$.

Each term in \eqref{eqn:Xpk_sumbound} can be bounded as follows:
\begin{align*}
  \expec{\abs{\zeta_{\balpha^1}} \cdots
	\abs{\zeta_{\balpha^k}} }
	&\leq {\infnorm{\u}^k} \left(\prod_{s=1}^k \prod_{l=1}^{a-1} B_{\alpha_l^s \alpha_{l+1}^s}\right)
		\expec{\prod_{s=1}^k \left(\prod_{l=1}^{a}\abs{h_{\alpha_l^s}} \right)} \\
	&\leq \infnorm{\u}^k \left(\prod_{s=1}^k \prod_{l=1}^{a-1} B_{\alpha_l^s \alpha_{l+1}^s}\right)
		\prod_{\gamma \in V(G)} \frac{1}{n},
\end{align*}
where the last step follows from property $(*)$ above and Definition~\ref{defn:moment-condns}.

Using the above, we can bound \eqref{eqn:Xpk_sumbound} as follows:
\begin{align*}
  {\expec{X_a^k}} &\leq \frac{\infnorm{\u}^k }{n^{v}} \sum_{\alpha_1,\cdots,\alpha_v}
	\left(\prod_{\set{\gamma,\gamma'}\in E(G)} B_{\alpha_{\gamma}\alpha_{\gamma'}}\right).
\end{align*}
where $v \defas \abs{V(G)}$ denotes the number of vertices in $G$.

Factorizing the above summation over different components of $G$, we obtain
\begin{align}\label{eqn:Xpk_sumbound2}
  {\expec{X_a^k}} &\leq \frac{\infnorm{\u}^k }{n^{v}} \prod_{j=1}^l \sum_{\alpha_1,\cdots,\alpha_{v_j}}
	\left(\prod_{\set{\gamma,\gamma'}\in E(G_j)} B_{\alpha_{\gamma}\alpha_{\gamma'}}\right),
\end{align}
where $l$ denotes the number of connected components of $G$, $G_j$ denotes the $j^{\textrm{th}}$ component of $G$, and
$v_j$ denotes the number of vertices in $G_j$.
We will now bound terms corresponding to one connected component at a time. Pick a connected component $G_j$. Since
$\alpha_1^s(1) = 1$ for every $s \in [a]$, we know that there exists a vertex $\alpha_{\gamma} \in G_j$ such that
$\alpha_{\gamma}(1) = 1$. Pick one such vertex as a root vertex and create a spanning tree $T_j$ of $G_j$. We use the
bound $B_{\alpha_{\gamma}\alpha_{\gamma'}} \leq 1$ for every $\set{\gamma,\gamma'} \in E_j \setminus T_j$.
The remaining summation $\sum_{\alpha_1,\cdots,\alpha_{v_j}} \left(\prod_{\set{\gamma,\gamma'}\in T_j} B_{\alpha_{\gamma}\alpha_{\gamma'}}\right)$
can be calculated bottom up from leaves to the root. Since
\begin{align*}
  \sum_{\alpha_{\gamma'}} B_{\alpha_{\gamma}\alpha_{\gamma'}} = n, \; \forall \; \gamma,
\end{align*}
we obtain
\begin{align*}
  \sum_{\alpha_1,\cdots,\alpha_{v_j}}
	\left(\prod_{\set{\gamma,\gamma'}\in E(G_j)} B_{\alpha_{\gamma}\alpha_{\gamma'}}\right)
	&\leq n^{v_j}.
\end{align*}
Plugging the above in \eqref{eqn:Xpk_sumbound2} gives us
\begin{align*}
  {\expec{X_a^k}} &\leq \frac{\infnorm{\u}^k }{n^{v}} n^{\sum_j v_j} = \infnorm{u}^k.
\end{align*}
Noting that the number of partitions $\Gamma$ is at most $(ka)^{ka}$, we obtain the bound
\begin{align*}
  {\expec{X_a^k}} &\leq \left({\infnorm{\u} \left(ka\right)^a}\right)^k.
\end{align*}

Choosing $k = 2\ceil{\frac{\log n}{a}} $ and applying $k^{\textrm{th}}$ moment Markov inequality, we obtain
\begin{align*}
\prob{\abs{X_a} > { \left(c \log n\right)^a \infnorm{\u}}} \leq {\expec{\abs{X_a}^k}} \left(\frac{1}{ \left(c \log n\right)^a \infnorm{\u}}\right)^k
	&\leq \left(\frac{ka}{c \log n}\right)^{ka} \leq n^{-2 \log \frac{c}{2}}.
\end{align*}
Going back to \eqref{eqn:split-up-low}, we have:
\begin{align*}
\prob{\abs{\widehat{X}_a} > { \left(c \log n\right)^a \infnorm{\u}}}
	&\leq 2^a \prob{\abs{{X}_a} > { \left(\frac{c}{2} \log n\right)^a \infnorm{\u}}} \\
	&\leq 2^a {\expec{\abs{X_a}^k}} \left(\frac{1}{ \left(\frac{c}{2} \log n\right)^a \infnorm{\u}}\right)^k \\
	&\leq 2^a \left(\frac{ka}{c \log n}\right)^{ka} \leq n^{-2 \log \frac{c}{4}}.
\end{align*}

Applying a union bound now gives us the result.


\section{Empirical Results}\label{sec:exp}
\begin{figure}[t]
  \centering
  \begin{tabular}[t]{cccc}
    \hspace*{-10pt}\includegraphics[width=.25\textwidth]{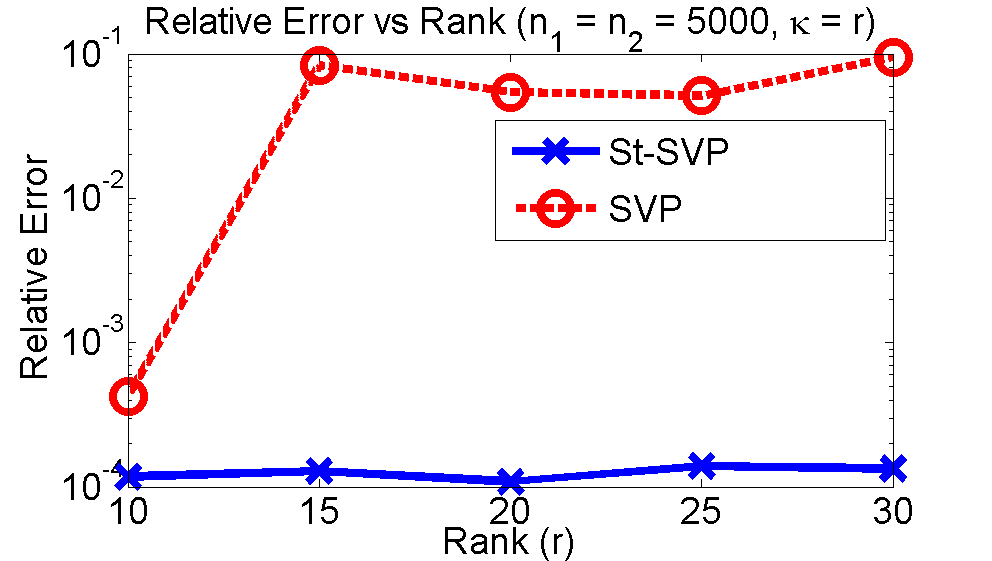}&
    \hspace*{-10pt}\includegraphics[width=.25\textwidth]{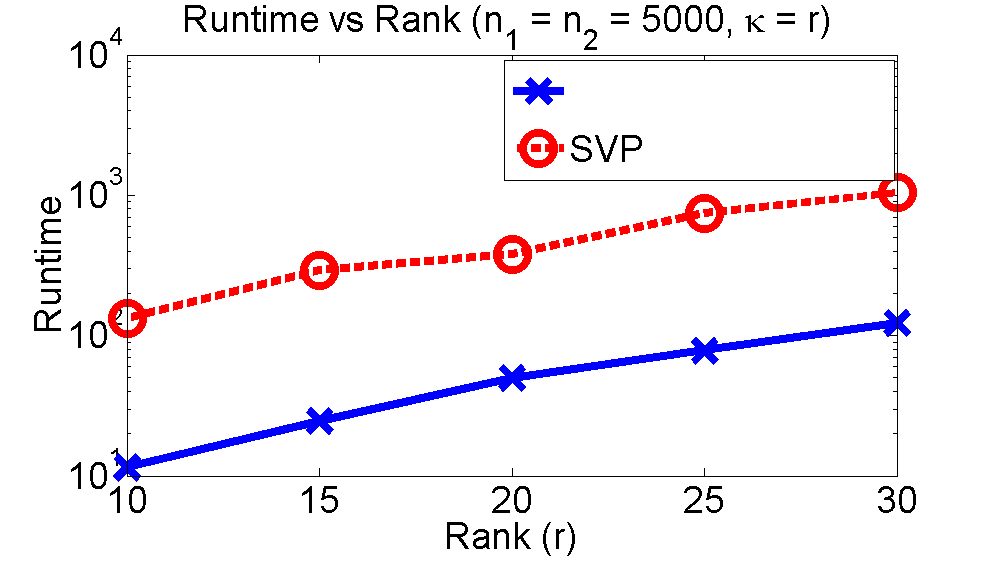}&
   \hspace*{-10pt} \includegraphics[width=.25\textwidth]{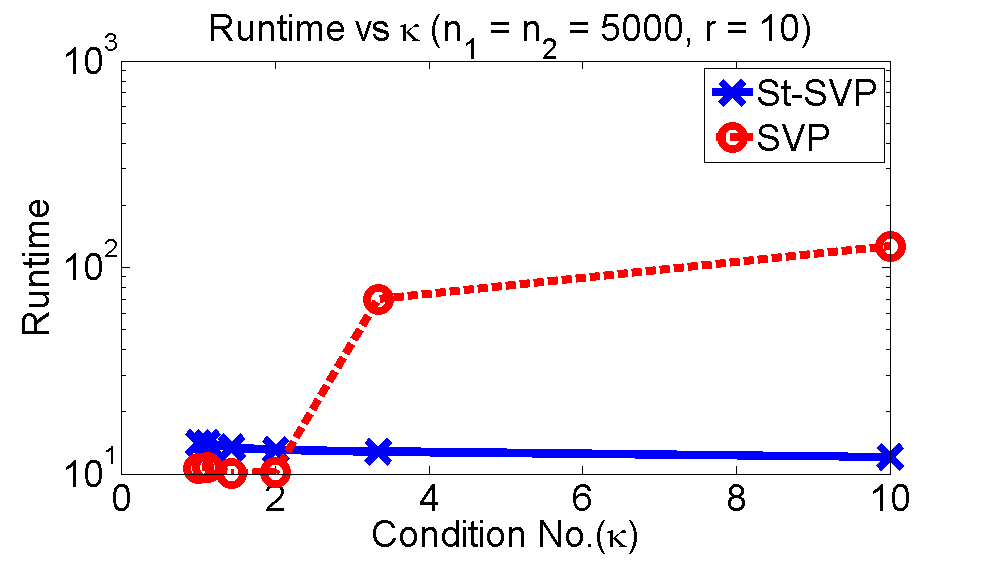}&
    \hspace*{-10pt}\includegraphics[width=.25\textwidth]{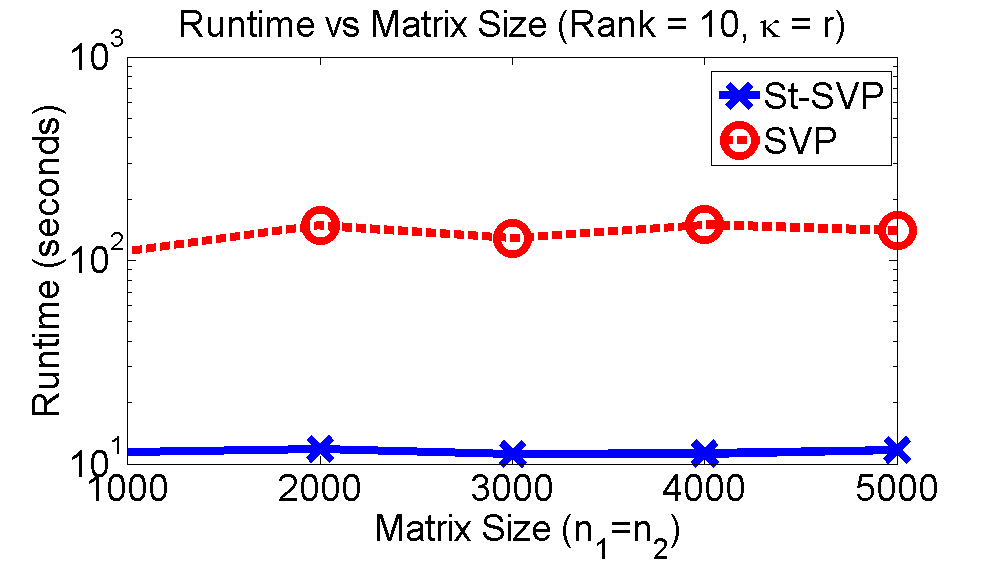}\\
(a)&(b)&(c)&(d)\vspace*{-5pt}
  \end{tabular}
  \caption{\svp~ vs \stsvp: simulations on synthetic datasets. a), b): recovery error and run time of the two methods
for varying rank. c): run time required by \stsvp~ and \svp~ with varying condition number.
d): run time of both the methods with varying matrix size.}
  \label{fig:exps}\vspace*{-2pt}
\end{figure}
In this section, we compare the performance of \stsvp~ with \svp~ on synthetic examples.
We do not however include comparison to other matrix completion methods like nuclear norm minimization
or alternating minimization; see \cite{JainMD10} for a comparison of \svp~ with those methods.

We implemented both the methods in Matlab and all the results are averaged over 5 random trials.
In each trial we generate a random low rank matrix
and observe $|\Omega|=5 (n_1+n_2) r \log(n_1+n_2)$ entries from it uniformly at random.

In the first experiment, we fix the matrix size ($n_1=n_2=5000$) and generate random matrices with varying rank $r$.
We choose the first singular value to be $1$ and the remaining ones to be $1/r$,
giving us a condition number of $\kappa=r$. Figure~\ref{fig:exps} (a) \& (b) show the error in recovery and the run time 
of the two methods, where we define the recovery error as $\twonorm{\widehat{\M}-\M} / \twonorm{\M}$.
We see that \stsvp~ recovers the underlying matrix much more accurately as compared to \svp.
Moreover, \stsvp~ is an order of magnitude faster than \svp. 

In the next experiment, we vary the condition number of the generated matrices. Interestingly, for small
$\kappa$, both \svp~ and \stsvp~ recover the underlying matrix in similar time. However, for larger $\kappa$,
the running time of \svp~ increases significantly and is almost two orders of magnitude larger than that of \stsvp. 
Finally, we study the two methods with varying matrix sizes while keeping all the other parameters fixed
($r=10$, $\kappa=1/r$). Here again, \stsvp~ is much faster than \svp.


\end{document}